\documentclass[10pt,english,reqno]{amsart} 
\usepackage[mathscr]{euscript}
\newcommand{\fr}{\mathfrak}
\newcommand{\cal}{\mathscr}

\newcommand{\op}{\operatorname}

\newcommand{\xysmall}{\xymatrix@C=1.5em@R=1.5em}

\newcommand{\Spec}{\mathrm{Spec}}

\newcommand{\dR}{\mathrm{dR}}
\newcommand{\Hom}{\mathrm{Hom}}

\newcommand{\QCoh}{\mathrm{QCoh}}
\newcommand{\Coh}{\mathrm{Coh}}

\newcommand{\Pic}{\mathrm{Pic}}

\newcommand{\et}{\textnormal{\'et}}

\newcommand{\Mod}{\textnormal{-}\mathrm{Mod}}
\newcommand{\Shv}{\textnormal{-}\mathrm{Shv}}

\newcommand{\PD}{\textnormal{-}\mathrm{PD}}

\newtheorem*{thm*}{Theorem}
\newtheorem{thm}{Theorem}[section]

\newtheorem{prop}[thm]{Proposition}
\newtheorem{lem}[thm]{Lemma}

\newtheorem{cor}[thm]{Corollary}

\theoremstyle{definition}

\newtheorem{claim}[thm]{Claim}

\newtheorem{eg-eg}[thm]{Example of Example}

\newtheorem{rem}[thm]{Remark}

\usepackage[colorlinks=false]{hyperref}
\usepackage{graphicx}
\usepackage{amssymb}
\usepackage{epstopdf}
\usepackage{enumerate}
\usepackage{tikz}
\usepackage{marginnote}
\usepackage{epigraph}

\usepackage{soul}

\usepackage{mdframed}
\newmdenv[
  topline=false,
  bottomline=false,
  rightline=false,
  skipabove=\topsep,
  skipbelow=\topsep
]{siderules}

\numberwithin{equation}{section}

\usepackage{mathtools}

\usepackage{epigraph}

\usepackage{color}
\usepackage[color,matrix,arrow]{xy}
\xyoption{all}

\title{$p$-torsion \'etale sheaves on the Jacobian of a curve}

\author{Yifei Zhao}
\email{yifei.zhao.93@gmail.com}
\date{November 5, 2021}

\begin{document}

\maketitle


\begin{abstract}
Suppose $X$ is a smooth, proper, geometrically connected curve over $\mathbb F_q$ with an $\mathbb F_q$-rational point $x_0$. For any $\mathbb F_q^{\times}$-character $\sigma$ of $\pi_1(X)$ trivial on $x_0$, we construct a functor $\mathbb L_n^{\sigma}$ from the derived category of coherent sheaves on the moduli space of deformations of $\sigma$ over the Witt ring $W_n(\mathbb F_q)$ to the derived category of constructible $W_n(\mathbb F_q)$-sheaves on the Jacobian of $X$. The functors $\mathbb L_n^{\sigma}$ categorify the Artin reciprocity map for geometric class field theory with $p$-torsion coefficients. We then give a criterion for the fully faithfulness of (an enhanced version of) $\mathbb L_n^{\sigma}$ in terms of the Hasse--Witt matrix of $X$.
\end{abstract}

\setcounter{tocdepth}{1}
\tableofcontents

\section*{Introduction}

\subsection{Results of this paper}
\subsubsection{}
Fix a prime $p$. Let $X$ be a smooth, proper, geometrically connected curve over $\mathbb F_q$, where $q$ is a power of $p$. The Artin reciprocity map of (unramified, global) geometric class field theory is a homomorphism of pro-finite groups:
\begin{equation}
\label{eq-cft-map}
\op{Pic}(\mathbb F_q) \rightarrow \pi_1(X)_{\op{ab}},
\end{equation}
where $\op{Pic}$ denotes the Picard scheme of $X$, and $\pi_1(X)_{\op{ab}}$ is the abelianization of the \'etale fundamental group. P.~Deligne constructed the homomorphism \eqref{eq-cft-map} by a geometric argument that relates $\ell$-adic characters of both sides, where $\ell$ is a prime different from $p$. Its key step lies in establishing an association:
\begin{equation}
\label{eq-cft-assoc}
\left\{\txt{rank--$1$ \'etale \\ $\overline{\mathbb Q}_{\ell}$-local systems on $X$}\right\} \rightarrow \left\{\txt{character $\overline{\mathbb Q}_{\ell}$-sheaves \\ on $\op{Pic}$}\right\}.
\end{equation}
The character $\overline{\mathbb Q}_l$-sheaf attached to a rank--$1$ \'etale local system satisfies a certain Hecke property, which ensures that its trace of Frobenius defines the correct character of $\op{Pic}(\mathbb F_q)$. (See \cite{To11} for an exposition of this approach.).

\subsubsection{}
In the case of $p$-torsion coefficients, this paper presents an alternative approach to constructing an analogue of the association \eqref{eq-cft-assoc}. It has the additional benefit of yielding some information about deformations of $p$-torsion characters of $\pi_1(X)$. This approach is categorical in nature and is based on the Fourier--Mukai--Laumon transform.

\subsubsection{} Let $W_n$ denote the length--$n$ Witt ring of $\mathbb F_q$ and we fix a smooth lift $X_n$ of the curve $X$ to $W_n$. For simplicity, we further assume that $X$ has an $\mathbb F_q$-rational point $x_0$, and we will only be concerned with $W_n^{\times}$-characters of $\pi_1(X)$ trivial on $\pi_1(x_0)$. (Here, $\pi_1(X)$ is defined by choosing a geometric point $\overline x_0$ lying over $x_0$.) Let $\op{Jac}$ denote the Jacobian scheme of $X$ classifying a degree--$0$ line bundle together with a trivialization at $x_0$.

\subsubsection{} The first goal of this paper is to define a (contravariant) functor, valued in the derived category of constructible \'etale $W_n$-sheaves on $\op{Jac}$:
\begin{equation}
\label{eq-artin-recip-enh-intro}
\mathbb L_{n} : \op{Coh}(\widetilde{\op{Jac}}{}_n^{\natural})_{\circ}^{\op{id}} \rightarrow W_n\Shv_c(\op{Jac})^{\op{op}}.
\end{equation}
Here, $\widetilde{\op{Jac}}{}_n^{\natural}$ is a certain scheme over $W_n$ whose sections are identified with rank--$1$ \'etale $W_n$-local systems on $X$ trivialized at $x_0$, with additional structure related to their deformations (see below). The left hand side means the derived category of coherent sheaves over $\widetilde{\op{Jac}}{}_n^{\natural}$ of finite Tor-dimension over $W_n$, equipped with an arbitrary automorphism.

\subsubsection{} We call $\mathbb L_n$ the $p$-torsion Artin reciprocity \emph{functor}. It categorifies the association:
$$
\left\{\txt{rank--$1$ \'etale $W_n$-local systems \\ on $X$ trivialized at $x_0$}\right\} \rightarrow \left\{\txt{character $W_n$-sheaves \\ on $\op{Jac}$}\right\}
$$
in the following sense: given an object of the left hand side, viewed as a section $\rho : \Spec(W_n) \rightarrow \widetilde{\op{Jac}}{}_n^{\natural}$, we may consider the pushforward $\rho_*\cal O_{W_n}$ (together with the identity automorphism) as an object of $ \op{Coh}(\widetilde{\op{Jac}}{}_n^{\natural})_{\circ}^{\op{id}}$. Then its image under $\mathbb L_n$ is a character $W_n$-sheaf on the Jacobian satisfying the Hecke eigen-property.

\subsubsection{} The main result of the paper concerns the behavior of the functor $\mathbb L_n$.

\begin{thm*}
The following are equivalent:
\begin{enumerate}[(a)]
	\item The functor $\mathbb L_n$ is fully faithful;
	\item The $\op{Fr}_X^*$-action on $\op H^1(X; \cal O_X)$ is nilpotent.
\end{enumerate}
\end{thm*}

\noindent
The $\op{Fr}_X^*$-action on $\op H^1(X; \cal O_X)$ is known as the Hasse--Witt matrix of $X$. For example, when $X$ has genus $1$, condition (b) means that $X$ is a supersingular elliptic curve.

\subsubsection{} We now turn to an overview of the construction of $\mathbb L_n$. Let $\op{Jac}_n$ denote the Jacobian of the lift $X_n$ and $\widetilde{\op{Jac}}_n$ be its universal additive extension, which classifies $\cal L\in\op{Jac}_n$ equipped with an integrable connection. Our construction uses the Fourier--Mukai--Laumon transform for crystalline $\cal D$-modules (i.e., quasi-coherent sheaves equipped with an integrable connection):
\begin{equation}
\label{eq-fml-transform-intro}
\Phi : \QCoh(\widetilde{\op{Jac}}_n) \xrightarrow{\sim} \cal D\Mod(\op{Jac}_n).
\end{equation}
Here, the two sides refer to the derived category of quasi-coherent sheaves, respectively $\cal D$-modules. We will verify a Frobenius compatibility statement of the functor $\Phi$, which asserts that the following diagram is commutative:
\begin{equation}
\label{eq-frob-compat-intro}
\xysmall{
	\QCoh(\widetilde{\op{Jac}}_n)  \ar[r]^-{\Phi} \ar[d]^{\op{Ver}_*} &  \cal D\Mod(\op{Jac}_n) \ar[d]^{\op{Fr}^*} \\
	\QCoh(\widetilde{\op{Jac}}_n)  \ar[r]^-{\Phi} &  \cal D\Mod(\op{Jac}_n)
}
\end{equation}
Here, the Frobenius-pullback of $\cal D$-modules is well defined \emph{without} a global lift of Frobenius, per an observation of P.~Berthelot \cite{Be00}. The Verschiebung endormophism on $\widetilde{\op{Jac}}_n$ is defined by Frobenius pullback of $\cal D$-modules on the curve, and we set $\widetilde{\op{Jac}}{}_n^{\natural}$ to be its fixed-point locus.

\subsubsection{} The functor $\mathbb L_n$ arises from taking ``categorical fixed points" of $\op{Ver}_*$ and $\op{Fr}^*$ on both sides of \eqref{eq-fml-transform-intro}. Namely, we use the $p$-torsion Riemann--Hilbert correspondence of Emerton--Kisin \cite{EK04} to relate Frobenius-fixed points in $\cal D\Mod(\op{Jac}_n)$ to constructible \'etale $W_n$-sheaves on $\op{Jac}$, and the Verschiebung-fixed point locus $\widetilde{\op{Jac}}{}_n^{\natural}$ to $W_n$-local systems on $X$.

It is worth emphasizing that the process of taking ``categorical fixed points" is $\infty$-category; namely, the analogous operation on the level of triangulated categories would yield an incorrect category. The interpretation of the Emerton--Kisin Riemann--Hilbert correspondence through categorical fixed points seems to be a fruitful idea, and we expect it to have other applications in the study of $p$-torsion constructible \'etale sheaves.

\subsubsection{} The present paper is motivated by certain perspectives from the geometric Langlands program in characteristic $p$. Indeed, for $n=1$, the transformation \eqref{eq-fml-transform-intro} generalizes to a fully faithful embedding of quasi-coherent sheaves on a dense open substack of $\op{LocSys}_r$, the stack of rank-$r$ crystalline local systems, into $\cal D$-modules on the stack of rank-$r$ vector bundles, by the work of Bezrukavnikov--Braverman \cite{BB06}.

Whether or not this embedding generalizes to quasi-coherent sheaves on the entire stack $\op{LocSys}_r$ remains speculative. However, if it does and satisfies the analogue of Frobenius compatibility \eqref{eq-frob-compat-intro}, techniques of the present paper would indicate a way to study mod-$p$ (or more generally, $p$-torsion) automorphic sheaves via Langlands dual data.

\subsection{Organization}

\subsubsection{} In the first section, we recall the $p$-torsion Riemann--Hilbert correspondence due to Emerton--Kisin. The key result is Proposition \ref{prop-ab-to-der} (and its Corollary) which allows us to view constructible $W_n$-sheaves as $\op{Fr}^*$-fixed points of the $\infty$-category of crystalline $\cal D$-modules.

\subsubsection{} In the second section, we verify that the Fourier--Mukai--Laumon transform satisfies Frobenius compatibility. This fact is then used to build the main functor $\mathbb L_n$. We then show that $\mathbb L_n$ de-categorifies into the Artin reciprocity map for $W_n^{\times}$-characters.

\subsubsection{} In the final section, we prove the main Theorem. The construction of $\mathbb L_n$ makes it clear that the only source of failure of fully faithfulness lies in the passage from modules on the Verschiebung-fixed locus $\widetilde{\op{Jac}}{}_n^{\natural}$ to modules on $\widetilde{\op{Jac}}_n$ equipped with an isomorphism to its Verschiebung-pushforward. This can be seen as a general phenomenon that arises in comparing geometric, \emph{vis-\`a-vis} categorial, fixed points. Our main result, Proposition \ref{prop-nilpotence-abstract}, answers this general question when the geometric fixed-point locus is a regularly immersed affine scheme.

\subsection{Notations}

\subsubsection{} We comment on our notations pertaining to category theory. The notations having to do with geometric objects will be explained at the beginning of each section.

\subsubsection{} In this paper, we use the theory of $\infty$-categories as developped by J.~Lurie \cite{Lu09} \cite{Lu17}. We do not make model-dependent arguments, so any flavor of $\infty$-categories with the same formal structure suffices for our purpose.

\subsubsection{} For an associative algebra $A$ (in the classical sense), we write $A\Mod$ for the $\infty$-category of chain complexes of $A$-modules. When we work with the abelian subcategory, we denote it by $A\Mod^{\heartsuit}$, understood as the heart of the usual $t$-structure on $A\Mod$. In line with this notation, we often write ``$A$-module'' to mean an object of $A\Mod$, i.e., a complex of $A$-modules.

\subsubsection{} The same convention applies to the $\infty$-categories $\QCoh(Y)$, $\cal D\Mod(Y)$, etc., for a scheme $Y$. By default, all functors are derived. However, sometimes we emphasize their derived nature using notations such as $Li^*$, $Rf_*$.

\subsection{Acknowledgements} The author thanks Robert Cass, Lin Chen, and Dennis Gaitsgory for conversations related to this work, and to Mark Kisin for answering a question about Katz's theorem. The author thanks the anonymous referee for helpful feedback on an earlier version of this paper. Special thanks are due to James Tao, who suggested several improvements.

\medskip

\section{Unit $F$-crystals}

In this section, we first recall the $p$-torsion Riemann--Hilbert correspondence due to Emerton--Kisin \cite{EK04}. It asserts that the derived category of $p$-torsion \'etale sheaves is (anti-)equivalent to that of arithmetic $\cal D$-modules $\cal M$ equipped with an isomorphism to its Frobenius pullback $\psi_{\cal M} : \cal M \xrightarrow{\sim} \op{Fr}^*\cal M$, satisfying some finiteness conditions. We then show that the datum of $\psi_{\cal M}$ can be expressed on the level of chain complexes, provided that one works with stable $\infty$-categories.

\subsection{The Riemann--Hilbert correspondence}
\label{sec-rh}

\subsubsection{}
In this section, we fix a perfect field $k$ containing $\mathbb F_q$. Let $W_n(k)$ (resp.~$W_n(\mathbb F_q)$) denote the ring of length--$n$ Witt vectors over $k$ (resp.~$\mathbb F_q$); it is equipped with a canonical lift of the $q$th power Frobenius, denoted by $\op{Fr}_{W_n(k)}$. Suppose $Y$ is a \emph{smooth} scheme over $W_n(k)$. Let $Y_0$ denote the base change of $Y$ to $k$.

\subsubsection{} Denote by $W_n(\mathbb F_q)\Shv(Y_0)$ the $\infty$-category of \'etale $W_n(\mathbb F_q)$-sheaves on $Y_0$. More precisely, it is the $\infty$-category of $W_n(\mathbb F_q)\Mod$-valued \'etale sheaves of \cite[Definition 2.2.1.2]{GL19}. It contains as full subcategory:
$$
W_n(\mathbb F_q)\Shv_c(Y_0) \subset W_n(\mathbb F_q)\Shv(Y_0)
$$
the $\infty$-category of constructible \'etale $W_n(\mathbb F_q)$-sheaves, which consists of objects $\cal F$ of finite Tor-dimension over $W_n(\mathbb F_q)$ and such that each $\op H^i(\cal F)$ is constructible (see \cite[\S2, Proposition-d\'efinition 4.6]{De77}).\footnote{The homotopy category of $W_n(\mathbb F_q)\Shv_c(Y)$ is denoted by $D^b_{\op{ctf}}(X_{\et}, W_n(\mathbb F_q))$ in \cite{EK04}.} The Riemann--Hilbert correspondence of Emerton--Kisin \cite{EK04} expresses $W_n(\mathbb F_q)\Shv_c(Y_0)$ in terms of \emph{unit $F$-crystals} on $Y$. The goal of this subsection is to review these objects and state a form of the Riemann--Hilbert correspondence.

\subsubsection{} For an integer $\nu\ge 0$, let $\cal D_Y^{(\nu)}$ denote the ring of $W_n(k)$-linear differential operators on $Y$ with level-$\nu$ divided powers; it can be thought of as the $\cal O_Y$-algebra generated by differential operators of the form $\partial^{p^k}/p^k!$ for $k\le\nu$. More formally, the ring $\cal D_Y^{(\nu)}$ is defined to be the topological dual of functions on $(Y\underset{W_n(k)}{\times} Y)_{\widehat Y}^{(\nu\PD)}$, the level-$\nu$ divided power neighborhood of the closed immersion $\Delta : Y\hookrightarrow Y\underset{W_n(k)}{\times}Y$ (see \cite[\S 2.2.1]{Be96}). Therefore, a quasi-coherent $\cal D_Y^{(\nu)}$-module is equivalent to a quasi-coherent $\cal O_Y$-module equivariant with respect to the action of the divided power infinitesimal groupoid:\footnote{This equivariance datum is called \emph{$\nu$-PD-stratification} in \emph{loc.cit}.}
$$
\xysmall{
\cdots \ar@<-1ex>[r] \ar[r] \ar@<1ex>[r] & (Y\underset{W_n(k)}{\times} Y)_{\widehat Y}^{(\nu\PD)} \ar@<-0.5ex>[r] \ar@<0.5ex>[r] & Y.
}
$$

\subsubsection{} There is a sequence of (non-injective) morphisms of filtered algebras:
$$
\cal D_Y^{(0)} \rightarrow \cal D_Y^{(1)} \rightarrow \cdots \rightarrow \cal D_Y:=\underset{\nu}{\op{colim}}\,\cal D_Y^{(\nu)}.
$$
In particular, there are forgetful functors in the opposite direction:
$$
\cal D_Y\Mod \rightarrow \cdots \rightarrow \cal D_Y^{(1)}\Mod \rightarrow \cal D_Y^{(0)}\Mod.
$$
Objects of $\cal D_Y^{(0)}\Mod$ are usually called \emph{crystalline} $\cal D$-modules.\footnote{We change the notation from the introduction, where $\cal D_Y^{(0)}$-modules was denoted by $\cal D_Y\Mod$.} They are precisely quasi-coherent $\cal O_Y$-modules equipped with an integrable connection relative to $W_n(k)$.

\subsubsection{}
\label{sec-frob-pullback}
An important observation due to P.~Berthelot \cite[\S2]{Be00} is that there is a canonically defined $W_n(\mathbb F_q)$-linear endofunctor $\op{Fr}_Y^*$ on $\cal D_Y^{(\nu)}\Mod^{\heartsuit}$, which can be constructed by choosing any local lift $\op{Fr}_Y$ of the $q$th power Frobenius and showing that the resulting $\cal D_Y^{(\nu)}$-module is canonically independent of the lift. Furthermore, $\op{Fr}_Y^*$ enhances to a functor (denoted by the same name):
\begin{equation}
\label{eq-berth-pullback}
\op{Fr}_Y^* : \cal D_Y^{(\nu)}\Mod^{\heartsuit} \rightarrow \cal D_Y^{(\nu+1)}\Mod^{\heartsuit}.
\end{equation}

\begin{rem}
In fact, the functor \eqref{eq-berth-pullback} is an equivalence of categories, by Th\'eor\`eme 2.3.6 of \emph{op.cit.}, but we shall not use this fact.
\end{rem}

\subsubsection{} We write $\cal D_{F, Y}^{(\nu)}\Mod^{\heartsuit}$ for the abelian category of objects $\cal M \in \cal D_Y^{(\nu)}\Mod^{\heartsuit}$ together with a morphism $\psi_{\cal M} : \op{Fr}_Y^* \cal M \rightarrow\cal M$. The datum of $\psi_{\cal M}$ can alternatively be described as an enhancement of the $\cal D_Y^{(\nu)}$-action to an action of:
$$
\cal D_{F, Y}^{(\nu)} := \bigoplus_{r\ge 0} (\op{Fr}_Y^*)^r\cal D_Y^{(\nu)},
$$
for a canonically defined algebra structure on $\cal D_{F, Y}^{(\nu)}$ (see \cite[\S13.3]{EK04}). We call $\cal D_{F, Y}^{(\nu)}\Mod^{\heartsuit}$ the abelian category of \emph{$F$-crystals of level $\nu$}. It contains enough injective objects and we may write $D^+(\cal D_{F, Y}^{(\nu)}\Mod^{\heartsuit})$ for its (bounded below) derived $\infty$-category.

\subsubsection{}
\label{sec-derived-cat-level}
We introduce two of its full subcategories:
$$
\cal D_{F, Y}^{(\nu)}\Mod^{\heartsuit}_{\op{lfg-u}} \subset \cal D_{F, Y}^{(\nu)}\Mod_{\op u}^{\heartsuit} \subset \cal D_{F, Y}^{(\nu)}\Mod^{\heartsuit}.
$$
They are defined by successively imposing the following conditions:
\begin{enumerate}[(a)]
	\item $\cal D_{F, Y}^{(\nu)}\Mod^{\heartsuit}_{\op u}$ --- where $\psi_{\cal M}$ is an isomorphism; these are called \emph{unit $F$-crystals}. We note that by \eqref{eq-berth-pullback}, the $\cal D_Y^{(\nu)}$-module structure is promoted to a $\cal D_Y$-module structure. Hence the notion of \emph{unit} $F$-crystals is independent of the level $\nu$, i.e. the forgetful functors:
	$$
	\cal D_{F, Y}\Mod^{\heartsuit}_{\op u} \rightarrow \cdots \rightarrow \cal D_{F,Y}^{(1)}\Mod^{\heartsuit}_{\op u} \rightarrow \cal D_{F, Y}^{(0)}\Mod^{\heartsuit}_{\op u}
	$$
	are equivalences of abelian categories.
	\item $\cal D_{F, Y}^{(\nu)}\Mod^{\heartsuit}_{\op{lfg-u}}$ --- where $\cal M$ further contains an $\cal O_Y$-coherent submodule $\cal M_0$ such that the action map $\cal D_{F, Y}^{(\nu)} \underset{\cal O_Y}{\otimes} \cal M_0 \rightarrow \cal M$ is surjective.
\end{enumerate}

\subsubsection{}
Write $D_{\op{lfg-u}, \circ}^b(\cal D_{F, Y}^{(\nu)}\Mod^{\heartsuit})$ for the full $\infty$-subcategory of $D^+(\cal D_{F, Y}^{(\nu)}\Mod^{\heartsuit})$ consisting of complexes $\cal M$ satisfying the following conditions:
\begin{enumerate}[(a)]
	\item $\cal M$ is of finite Tor-dimension as a complex of $\cal O_Y$-modules;
	\item $\op H^i(\cal M)$ belongs to $\cal D_{F, Y}^{(\nu)}\Mod^{\heartsuit}_{\op{lfg-u}}$ for each $i$.
\end{enumerate}
The functor \eqref{eq-berth-pullback} can be promoted to $\op{Fr}_Y^* : \cal D_{F,Y}^{(\nu)}\Mod^{\heartsuit} \rightarrow \cal D_{F,Y}^{(\nu+1)}\Mod^{\heartsuit}$ and thus on the derived $\infty$-category. In particular, we see that $D_{\op{lfg-u}, \circ}^b(\cal D_{F, Y}^{(\nu)}\Mod^{\heartsuit})$ is also independent of the level $\nu$, i.e., the following forgetful functors:
\begin{equation}
\label{eq-equiv-change-level}
D_{\op{lfg-u}, \circ}^b(\cal D_{F, Y}\Mod^{\heartsuit}) \rightarrow \cdots\rightarrow D_{\op{lfg-u}, \circ}^b(\cal D_{F, Y}^{(1)}\Mod^{\heartsuit}) \rightarrow D_{\op{lfg-u}, \circ}^b(\cal D_{F, Y}^{(0)}\Mod^{\heartsuit})
\end{equation}
are equivalences.

\subsubsection{}
Define a functor of $W_n(\mathbb F_q)$-linear $\infty$-categories (here, we take $\nu$ to be infinite):
$$
\op{Sol} : D^b(\cal D_{F,Y}\Mod^{\heartsuit}) \rightarrow W_n(\mathbb F_q)\Shv^+(Y_0),
$$
by sending $\cal M$ to:
$$
\op{Sol}(\cal M) := R\Hom_{\cal D_{F, Y}}(\cal M_{\et}, \cal O_Y)[\dim(Y)],
$$
where $\cal M_{\et}$ is the \'etale $\cal D_{F, Y}$-module associated to $\cal M$.

\begin{thm}[Emerton--Kisin]
The restriction of $\op{Sol}$ to $D_{\op{lfg-u}, \circ}^b(\cal D_{F,Y}\Mod^{\heartsuit})$ defines an anti-equivalence of $W_n(\mathbb F_q)$-linear $\infty$-categories:
\begin{equation}
\label{eq-rh}
D_{\op{lfg-u}, \circ}^b(\cal D_{F,Y}\Mod^{\heartsuit})^{\op{op}} \xrightarrow{\sim} W_n(\mathbb F_q)\Shv_c(Y_0).
\end{equation}
\end{thm}
\begin{proof}
It suffices to check that $\op{Sol}$ becomes an equivalence after passing to the homotopy categories. For $q=p$, this is the main theorem of \cite{EK04}, but their proof also applies to the case where $q$ is a power of $p$.
\end{proof}

\begin{rem}
The equivalences \eqref{eq-equiv-change-level} show that the same result holds for $\cal D_{F,Y}^{(\nu)}$-modules of any level $\nu\ge 0$.
\end{rem}

\subsection{Katz's equivalence}

\subsubsection{} As in the characteristic--zero case, \eqref{eq-rh} specializes to an equivalence of $1$-categories corresponding to \emph{lisse} sheaves. The latter equivalence is essentially due to N.~Katz \cite{Ka73}. We will need a generalization of this special case where we allow deformations over $W_n(\mathbb F_q)$. The result is likely well known, but we could not find a suitable reference.

\begin{lem}
\label{lem-katz-equiv}
Suppose $S=\Spec(A)$ is a finite $W_n(\mathbb F_q)$-scheme. Then the following categories are equivalent functorially in $S$:
\begin{enumerate}[(a)]
	\item \'etale $A$-sheaves on $Y$ which are locally isomorphic to $\underline A^{\oplus r}$;
	\item rank--$r$ vector bundles $\cal E$ on $S\underset{W_n(\mathbb F_q)}{\times}Y$ equipped with an integrable connection along $Y$ relative to $W_n(k)$, and an isomorphism $\op{Fr}_Y^*\cal E \xrightarrow{\sim} \cal E$ as $\cal D^{(0)}$-modules.
\end{enumerate}
\end{lem}

\noindent
Here, the symbol $\cal D^{(0)}$ stands for the level--$0$ differential operators on $S\underset{W_n(\mathbb F_q)}{\times}Y$ relative to $S\underset{W_n(\mathbb F_q)}{\times}W_n(k)$.

\subsubsection{} The case $A=W_n(\mathbb F_q)$ is essentially Proposition 4.1.1 of \emph{loc.cit.}, the only difference being that Katz's equivalence uses a global lift of the Frobenius on $Y$ whereas we use $\cal D^{(0)}$-modules and their canonical Frobenius pullbacks. We will explain how the proof of the cited Proposition can be adapted to yield a proof of Lemma \ref{lem-katz-equiv}.

\subsubsection{} One first defines a functor $F$ from (a) to (b). Let $\cal F$ be an object of the category (a). By passing to an \'etale cover of $Y$, we may assume that $\cal F$ is trivialized. We will then construct an object of (b) while remembering its descent datum. As a quasi-coherent $\cal O_Y$-module, we set:\footnote{We note that since $k$ is a perfect field, $W_n(k)$ is flat over $W_n(\mathbb F_q)$. Thus $\cal O_Y$ is also flat over $W_n(\mathbb F_q)$ and the tensor product can be safely taken in the classical sense.}
$$
\cal E := \cal F \underset{W_n(\mathbb F_q)}{\otimes}\cal O_Y.
$$
The integrable connection and Frobenius structure on $\cal E$ are both defined by actions on the $\cal O_Y$-factor. Finally, the $A$-module structure on $\cal F$ induces an $A$-action on $\cal E$. This turns $\cal E$ into a quasi-coherent sheaf on $S\underset{W_n(\mathbb F_q)}{\times}Y$, which is locally free of rank $r$.

\subsubsection{}
To show that this functor is an equivalence is a local question on $Y$. Hence we may assume that $Y$ is equipped with a lift $\op{Fr}_Y$ of the $q$th power Frobenius endomorphism on its special fiber. We then define another category:
\begin{enumerate}[(b')]
	\item vector bundles $\cal E$ on $S\underset{W_n(\mathbb F_q)}{\times}Y$ equipped with an isomorphism $\psi : \op{Fr}_Y^*\cal E \xrightarrow{\sim} \cal E$.
\end{enumerate}

\begin{claim}
The forgetful functor from (b) to (b') is an equivalence.
\end{claim}
\begin{proof}
Recall that for an integer $N\gg 0$ (depending only on $Y$) and any $\cal O_Y$-module $\cal M$, the quasi-cherent sheaf $(\op{Fr}_Y^*)^N\cal M$ acquires a \emph{canonical} connection $\nabla_{\op{can}}$, defined functorially in $\cal M$. Furthermore, if $\cal M$ already admits a connection, then $\nabla_{\op{can}}$ identifies with its pullback. This shows that an object $\cal E$ in (b') acquires a connection, making the isomorphism:
$$
\psi^N : (\op{Fr}_Y^*)^N \cal E \xrightarrow{\sim} \cal E
$$
$\cal D^{(0)}$-equivariant. To show that this connection upgrades $\cal E$ to an object of (b), we must show that the structural morphism $\psi$ is itself $\cal D^{(0)}$-equivariant. However, this follows from considering the following commutative diagram:
$$
\xymatrix@R=1.5em{
	(\op{Fr}_Y^*)^{N+1}\cal E \ar[r]^-{(\op{Fr}_Y^*)^N\psi}\ar[d]^{\op{Fr}_Y^*\psi^N} & (\op{Fr}_Y^*)^N \cal E \ar[d]^{\psi^N} \\
	\op{Fr}_Y^*\cal E \ar[r]^-{\psi} & \cal E
}
$$
where all arrows besides $\psi$ are already known to be $\cal D^{(0)}$-equivariant.
\end{proof}

\subsubsection{} Thus we have reduced to showing the functor $F'$ from (a) to (b'), given by composing $F$ with the forgetful functor, is an equivalence. It admits an \emph{a priori} partially defined right adjoint $(F')^R$, given by sending $(\cal E, \psi)$ (regarded as a quasi-coherent $\cal O_Y$-module with Frobenius structure $\psi$) to the fixed points of the Frobenius-twisted linear morphism:
$$
\varphi : \cal E \hookrightarrow \op{Fr}_Y^*\cal E \xrightarrow{\psi} \cal E.
$$
We denote the resulting \'etale $W_n(\mathbb F_q)$-module by $\cal E^{\varphi\text{-fixed}}$; the $A$-action on $\cal E$ induces one on $\cal E^{\varphi\text{-fixed}}$, making it an \'etale $A$-sheaf. This functor is well-defined precisely when $\cal E^{\varphi\text{-fixed}}$ is locally isomorphic to a product of the constant $A$-sheaf.

\subsubsection{} The following claim shows that $(F')^R$ is well-defined on the essential image of $F'$ and that $F'$ is fully faithful.

\begin{claim}
For any finite, \'etale $W_n(\mathbb F_q)$-sheaf $\cal F$, the following natural map is bijective:
$$
\cal F \rightarrow (\cal F \underset{W_n(\mathbb F_q)}{\otimes} \cal O_Y)^{\varphi\text{-fixed}}.
$$
\end{claim}
\begin{proof}
By the classification of finite modules over a PID (applied to $\cal F$ viewed as a $W(\mathbb F_q)$-module), we obtain an isomorphism $\cal F\xrightarrow{\sim} \bigoplus_j \cal F_j$ where $\cal F_j$ is a \emph{free} $W_j(\mathbb F_q)$-module for $1\le j\le n$. Therefore, we have:
$$
\cal F \underset{W_n(\mathbb F_q)}{\otimes} \cal O_Y \xrightarrow{\sim} \bigoplus_{1\le j\le n} \cal F_j \underset{W_j(\mathbb F_q)}{\otimes} \cal O_{Y_j}.
$$
Here, $Y_j$ denotes the base change of $Y$ to $W_j(\mathbb F_q)$. The result follows, since $W_j(\mathbb F_q)$ identifies with the $\varphi$-fixed elements of $\cal O_{Y_j}$.
\end{proof}

\subsubsection{} Finally, we show that $F'$ is essentially surjective. Given any $(\cal E, \psi)$ in (b'), this amounts to the following problem:
\begin{itemize}
	\item \'Etale locally on $Y$, construct a basis $\underline h$ of $\cal E$ (over $S\underset{W_n(\mathbb F_q)}{\times} Y$) such that $\varphi$ fixes $\underline h$.
\end{itemize}
The argument is then essentially copying Katz's proof. We procede by induction on $n\ge 1$.

\subsubsection{} For $n=1$, the vector bundle $\cal E$ is \emph{also} locally free as an $\cal O_Y$-module. By \emph{loc.cit.}, we may functorially assign a finite, locally constant \'etale $\mathbb F_q$-sheaf $\cal F$ on $Y$ together with an isomorphism $\cal F \underset{\mathbb F_q}{\otimes} \cal O_Y\xrightarrow{\sim} \cal E$. The $A$-action on $\cal E$ is inherited by $\cal F$, making the above isomorphism $A$-equivariant. The hypothesis that $\cal E$ is locally free on $S \underset{\mathbb F_q}{\times} Y$ then shows that $\cal F$ is locally constant as an $A$-module. This gives a $\varphi$-fixed basis of $\cal E$.\footnote{In fact, it proves the essential surjectivity directly.}

\subsubsection{} Assuming that a $\varphi$-fixed basis exists for $n-1$, we will lift it to a basis $\underline h$ of $\cal E$ over $S\underset{W_n(\mathbb F_q)}{\times} Y$; it is not necessarily fixed by $\varphi$. In order to adjust it into a $\varphi$-fixed basis, we apply the argument in p.145 of \emph{loc.cit.}. It reduces the problem to solving $r^2$ equations in $A_1 
\underset{\mathbb F_q}{\otimes} \cal O_Y$ (where $A_1$ is the reduction of $A$ mod $p$) of the following type, with $\delta$'s being given:
\begin{equation}
\label{eq-artin-schreier}
\varphi(e) + \delta = e.
\end{equation}
Here, $\varphi$ denotes the Frobenius action on the $\cal O_{Y_1}$-factor. We write $\delta = \sum_{i} a_i \otimes f_i$. The equation \eqref{eq-artin-schreier} then reduces to finding solutions $x_i \in \cal O_{Y_1}$ to the equations $x_i^q + f_i = x_i$ for all $i$. Indeed, having the $x_i$'s, we may then set $e := \sum_i a_i\otimes x_i$. These Artin--Schreier equations are simultaneously solved over a finite \'etale cover of $Y_n$. \qed(Lemma \ref{lem-katz-equiv})

\subsubsection{}
\label{sec-katz-to-ek} The particular case of Lemma \ref{lem-katz-equiv} where $A=W_n(\mathbb F_q)$ is captured by the Emerton--Kisin Riemann--Hilbert correspondence \eqref{eq-rh} in the following manner. We let $\op{LocSys}(Y)$ and $\op{Conn}(Y)^{\op{Fr}^*}$ denote the two categories appearing in Lemma \ref{lem-katz-equiv} for $A = W_n(\mathbb F_q)$. Then the following diagram commutes:

$$
\xymatrix@R=1.5em@C=5em{
	\op{Conn}(Y)^{\op{Fr}^*} \ar[r]^-{\text{Lemma \ref{lem-katz-equiv}}}\ar[d]^{\mathbb D} & \mathrm{LocSys}(Y) \ar[dd] \\
	(\op{Conn}(Y)^{\op{Fr}^*})^{\op{op}} \ar[d] & \\
	D_{\op{lfg-u}, \circ}^b(\cal D_{F,Y}^{(0)}\Mod^{\heartsuit})^{\op{op}} \ar[r]^-{\op{Sol}[-\dim(Y)]} & W_n(\mathbb F_q)\Shv_c(Y)
}
$$
Here, the unlabeled vertical arrows are the tautological inclusions; the functor $\mathbb D$ is the operation of taking the monoidal dual $\cal E\leadsto \cal E^{\vee}$ with its induced connection and Frobenius-structure.

\begin{rem}
The duality functor $\mathbb D$ cannot be extended to $D_{\op{lfg-u}, \circ}^b(\cal D_{F,Y}\Mod^{\heartsuit})$, so we do not obtain a covariant Riemann--Hilbert equivalence. However, a version of covariant Riemann--Hilbert equivalence \emph{does} exist. It appeals to a different notion than $F$-crystals, and is the subject of the work of Bhatt--Lurie \cite{BL17}.
\end{rem}

\subsection{Categorical fixed points}
\label{sec-cat-fix}

\subsubsection{} Let $\cal C$ be a stable $\infty$-category together an endomorphism $F : \cal C\rightarrow\cal C$. The \emph{fixed point $\infty$-category} $\cal C^F$ is defined as the equalizer of $F$ with $\op{id}_{\cal C}$. Equivalenty, $\cal C^F$ is a fiber product of stable $\infty$-categories:
$$
\xymatrix@R=1.5em{
	\cal C^F \ar[r]\ar[d] & \cal C \ar[d]^{\Delta} \\
	\cal C \ar[r]^-{(F,\op{id})} & \cal C\times\cal C
}
$$
An object of $\cal C^F$ is an object $c\in\cal C$ together with an isomorphism $Fc\xrightarrow{\sim} c$.

\subsubsection{}
\label{sec-laxed-fixed-points}
More generally, one can define a lax version $\cal C^{F\rightarrow\op{id}}$ of objects $c\in\cal C$ together with a morphism $Fc\rightarrow c$, so that $\cal C^F$ is the full $\infty$-subcategory of $\cal C^{F\rightarrow\op{id}}$ where the structural morphism is an isomorphism. Formally, $\cal C^{F\rightarrow\op{id}}$ is the fiber product:
$$
\xysmall{
	\cal C^{F\rightarrow\op{id}} \ar[r]\ar[d] & \cal C^{\Delta^1} \ar[d]^{\op{can}} \\
	\cal C \ar[r]^-{(F, \op{id})} & \cal C\times\cal C
}
$$
where $\cal C^{\Delta^1}:=\op{Funct}(\Delta^1, \cal C)$ is the functor $\infty$-category from the interval, and the canonical map sends $d\rightarrow c$ to $(d,c)$.

\begin{rem}
The $\infty$-category $\cal C^{F\rightarrow\op{id}}$ has an obvious variant $\cal C^{\op{id}\rightarrow F}$ consisting of objects $c\in\cal C$ together with a morphism $c\rightarrow Fc$.
\end{rem}

\subsubsection{} The following fact shows how to calculate Hom-spaces in $\cal C^{F\rightarrow\op{id}}$ (hence also in $\cal C^F$).

\begin{lem}
\label{lem-fixed-point-hom}
Given objects $Fc_1 \xrightarrow{\psi_1} c_1$, $Fc_2\xrightarrow{\psi_2} c_2$ in $\cal C^{F\rightarrow\op{id}}$, their $\op{Hom}$-space in $\cal C^{F\rightarrow\op{id}}$ is the equalizer of:
$$
\xymatrix{
	 \op{Hom}_{\cal C}(c_1, c_2) \ar@<0.5ex>[r]^-{\psi_2\circ F(-)} \ar@<-0.5ex>[r]_-{(-)\circ\psi_1} & \op{Hom}_{\cal C}(Fc_1, c_2).
}
$$
\end{lem}
\begin{proof}
The Hom-space in a limit of $\infty$-categories is the limit of Hom-spaces. Therefore, the problem reduces to showing that Hom-spaces in $\cal C^{\Delta^1}$ identify with the fiber product:
$$
\xysmall{
	\op{Hom}_{\cal C^{\Delta^1}}(d_1\rightarrow c_1, d_2\rightarrow c_2) \ar[r]\ar[d] & \op{Hom}_{\cal C}(d_1, d_2) \ar[d] \\
	\op{Hom}_{\cal C}(c_1, c_2) \ar[r] & \op{Hom}_{\cal C}(d_1, c_2)
}
$$
This is in turn a particular case of the calculation of natural transformations as limit over the twisted arrow category.
\end{proof}

\subsubsection{} Suppose $\cal C$ is equipped with a $t$-structure and $F$ is \emph{$t$-exact}. Then $\cal C^{F\rightarrow\op{id}}$ inherits a $t$-structure, for which the forgetful functor to $\cal C$ is $t$-exact.

\begin{lem}
\label{lem-right-complete}
If $\cal C$ is right (resp.~left) complete with respect to the $t$-structure, then the same holds for $\cal C^{F\rightarrow\op{id}}$.
\end{lem}
\begin{proof}
Right (resp.~left) completeness of $\cal C$ means that the canonical functor $\cal C \rightarrow \lim_{n\rightarrow\infty} \cal C^{\le n}$ (resp.~$\cal C \rightarrow \lim_{n\rightarrow -\infty}\cal C^{\ge n}$), where the connecting functors $\cal C^{\le n+1} \rightarrow \cal C^{\le n}$ (resp.~$\cal C^{\ge n-1} \rightarrow \cal C^{\ge n}$) are given by truncations $\tau^{\le n}$ (resp.~$\tau^{\ge n}$), is an equivalence. Since $(\cal C^{F \rightarrow \op{id}})^{\le n}$ is identified with the fiber product:
$$
\xysmall{
	(\cal C^{F\rightarrow \op{id}})^{\le n} \ar[r]\ar[d] & (\cal C^{\le n})^{\Delta^1} \ar[d]^{\op{can}} \\
	\cal C^{\le n} \ar[r]^-{(F, \op{id})} & \cal C^{\le n}\times\cal C^{\le n}
}
$$
in a way that is compatible with the truncation functors, we see that $\cal C^{F \rightarrow\op{id}}$ is identified with $\lim_{n\rightarrow\infty} (\cal C^{F\rightarrow\op{id}})^{\le n}$. Under the appropriate hypotheses, the same argument yields the left-completeness and the corresponding statements for $\cal C^F$.
\end{proof}

\subsection{$F$-crystals as categorial fixed points}
\label{sec-f-crys-fixed-points}

\subsubsection{} We return to the setting of \S\ref{sec-rh}. Note that the functor \eqref{eq-berth-pullback} enhances to a $W_n(\mathbb F_q)$-linear functor of $\infty$-categories:
\begin{equation}
\label{eq-fr-equiv}
\op{Fr}_Y^* : \cal D_Y^{(\nu)}\Mod \rightarrow \cal D_Y^{(\nu+1)}\Mod.
\end{equation}
Its composition with the forgetful functor to $\cal D_Y^{(\nu)}\Mod$ will still be denoted by $\op{Fr}_Y^*$. Let $\cal D_Y^{(\nu)}\Mod^{\op{Fr}^*}$ (resp.~$\cal D_Y^{(\nu)}\Mod^{\op{Fr}^*\rightarrow\op{id}}$) denote the (resp.~lax) $\op{Fr}_Y^*$-fixed point $\infty$-category of $\cal D_Y^{(\nu)}\Mod$. We note that the equivalence \eqref{eq-fr-equiv} shows that $\cal D_Y^{(v)}\Mod^{\op{Fr}^*}$ is independent of the level $\nu$, i.e., the forgetful functors:
$$
\cal D_Y\Mod^{\op{Fr}^*} \rightarrow \cdots \rightarrow \cal D_Y^{(1)}\Mod^{\op{Fr}^*} \rightarrow \cal D_Y^{(0)}\Mod^{\op{Fr}^*}
$$
are equivalences, with inverses given by $\op{Fr}_Y^*$.

\subsubsection{}
We note that $\op{Fr}_Y^*$ is $t$-exact with respect to the tautological $t$-structure on $\cal D_Y^{(\nu)}\Mod$; indeed, this is because $Y$ is smooth, so any local lift of the Frobenius map is flat. Thus $\cal D_Y^{(\nu)}\Mod^{\op{Fr}^*\rightarrow \op{id}}$ inherits a $t$-structure, for which the forgetful functor to $\cal D_Y^{(\nu)}\Mod$ is $t$-exact. The following is immediate:

\begin{lem}
\label{lem-frob-module-heart}
The heart $(\cal D_Y^{(\nu)}\Mod^{\op{Fr}^*\rightarrow\op{id}})^{\heartsuit}$ identifies with $\cal D_{F,Y}^{(\nu)}\Mod^{\heartsuit}$. \qed
\end{lem}

\noindent
Furthermore, in order the check that an object of $\cal D_Y^{(\nu)}\Mod^{\op{Fr}^*\rightarrow\op{id}}$ belongs to the full $\infty$-subcategory $\cal D_Y^{(\nu)}\Mod^{\op{Fr}^*}$, it suffices to check that that its cohomology objects belongs to $\cal D_{F,Y}^{(\nu)}\Mod^{\heartsuit}_{\op u}$, i.e., are \emph{unit} $F$-crystals.

\subsubsection{} It follows from Lemma \ref{lem-frob-module-heart} that we have a canonically defined, $t$-exact functor:
\begin{equation}
\label{eq-ab-to-der}
D^+(\cal D_{F,Y}^{(\nu)}\Mod^{\heartsuit}) \rightarrow (\cal D_Y^{(\nu)}\Mod^{\op{Fr}^*\rightarrow\op{id}})^+.
\end{equation}
We specify the full $\infty$-subcategory $(\cal D_Y^{(\nu)}\Mod^{\op{Fr}^*})_{\op{lfg-u},\circ} \subset (\cal D_Y^{(\nu)}\Mod^{\op{Fr}^*\rightarrow\op{id}})^+$ as consisting of objects of finite Tor-dimension over $\cal O_Y$, whose cohomology groups belong to $\cal D_{F,Y}^{(\nu)}\Mod^{\heartsuit}_{\op{lfg-u}}$. In particular, the structural morphism of objects $\op{Fr}^*_Y\cal M\rightarrow\cal M$ in this full subcategory are isomorphisms. The same argument as in \S\ref{sec-frob-pullback} shows that the $\infty$-category $(\cal D_Y^{(\nu)}\Mod^{\op{Fr}^*})_{\op{lfg-u},\circ}$ is defined independently of the level $\nu$.

\subsubsection{} The following Proposition will be proved in the next subsection:

\begin{prop}
\label{prop-ab-to-der}
The functor \eqref{eq-ab-to-der} is an equivalence of $\infty$-categories.
\end{prop}

\begin{cor}
\label{cor-ekrh}
There is a canonical (anti-)equivalence of $W_n(\mathbb F_q)$-linear $\infty$-categories:
$$
(\cal D_Y^{(\nu)}\Mod^{\op{Fr}^*})_{\op{lfg-u},\circ}^{\op{op}} \xrightarrow{\sim} W_n(\mathbb F_q)\Shv_c(Y_0).
$$
\end{cor}
\begin{proof}
It follows from Proposition \ref{prop-ab-to-der} that we have an equivalence:
\begin{align*}
(\cal D_Y^{(\nu)}\Mod^{\op{Fr}^*})_{\op{lfg-u},\circ} \xrightarrow{\sim} & D_{\op{lfg-u}, \circ}^b(\cal D_{Y,F}^{(\nu)}\Mod^{\heartsuit}) \\
& \xrightarrow{\sim} D_{\op{lfg-u}, \circ}^b(\cal D_{Y,F}\Mod^{\heartsuit}) \xrightarrow{\sim} W_n(\mathbb F_q)\Shv_c(Y_0)^{\op{op}},
\end{align*}
where the last step is the Riemann--Hilbert correspondence \eqref{eq-rh}.
\end{proof}

\begin{rem}
If we worked with triangulated categories instead of stable $\infty$-categories, the resulting fixed point category would be defined incorrectly. This is the main reason we use $\infty$-categories in this paper.
\end{rem}

\subsection{Proof of Proposition \ref{prop-ab-to-der}}

\subsubsection{}
\label{sec-crystal-adjunct} We first recall a basic adjunction for $F$-crystals:
\begin{equation}
\label{eq-ind-crys}
\xymatrix{
	\cal D_Y^{(\nu)}\Mod^{\heartsuit} \ar@/^1pc/[r]^{\cal D_{F,Y}^{(\nu)}\underset{\cal D_Y}{\otimes}-} & \cal D_{F,Y}^{(\nu)}\Mod^{\heartsuit} \ar@/^1pc/[l]^{\op{oblv}}
}
\end{equation}
Since $\cal D_{F,Y}^{(\nu)}$ is flat as a right $\cal D_Y^{(\nu)}$-module (in fact, locally free), the left adjoint is exact. In particular, the right adjoint preserves injective objects. Following the terminology of \cite[\S13.6]{EK04}, we call a $\cal D_{F,Y}^{(\nu)}$-module \emph{induced}, if it arises from applying the functor $\cal D_{F,Y}^{(\nu)}\underset{\cal D_Y^{(\nu)}}{\otimes}-$ to a $\cal D_Y^{(\nu)}$-module.

\subsubsection{} Each $F$-crystal $(\cal M, \psi_{\cal M})$ admits a canonical two-step resolution by induced $F$-crystals (Proposition 13.6.1 of \emph{loc.cit.}):
\begin{equation}
\label{eq-two-step-filt}
0 \rightarrow \cal D_{F,Y}^{(\nu)} \underset{\cal D_Y^{(\nu)}}{\otimes} \op{Fr}_Y^*\cal M \xrightarrow{\alpha} \cal D_{F,Y}^{(\nu)} \underset{\cal D_Y^{(\nu)}}{\otimes} \cal M \rightarrow \cal M \rightarrow 0,
\end{equation}
where $\alpha$ is the \emph{difference} between the following two maps:
\begin{enumerate}[(a)]
	\item The adjoint of the map of $\cal D_Y^{(\nu)}$-modules:
	$$
	\op{Fr}_Y^*\cal M \xrightarrow{\sim} \op{Fr}_Y^*\cal D_Y^{(\nu)} \underset{\cal D_Y^{(\nu)}}{\otimes} \cal M \hookrightarrow \cal D_{F,Y}^{(\nu)} \underset{\cal D_Y^{(\nu)}}{\otimes} \cal M.
	$$
	\item The adjoint of the map of $\cal D_Y^{(\nu)}$-modules:
	$$
	\op{Fr}_Y^*\cal M \xrightarrow{\psi_{\cal M}} \cal M \hookrightarrow \cal D_{F,Y}^{(\nu)} \underset{\cal D_Y^{(\nu)}}{\otimes} \cal M.
	$$
\end{enumerate}

\begin{rem}
Although $\alpha$ is a map between induced $F$-crystals, it is not induced from a map of $\cal D_Y^{(\nu)}$-modules.
\end{rem}

\subsubsection{} Suppose $\cal C$ is a stable $\infty$-category equipped with a right complete $t$-structure, and $\cal C^{\heartsuit}$ has enough injective objects. Then the canonical functor $D^+(\cal C^{\heartsuit}) \rightarrow \cal C^+$ is an equivalence if and only if the following conditions holds:
\begin{itemize}
	\item For every injective object $I \in \cal C^{\heartsuit}$, the abelian group $\op{Ext}^i_{\cal C}(\cal M, I) = 0$ for all $\cal M\in\cal C^{\heartsuit}$ and $i\ge 1$.
\end{itemize}
This is essentially (the dual of) \cite[Proposition 1.3.3.7]{Lu17}.

\subsubsection{} We now combine the above ingredients to give a proof of Proposition \ref{prop-ab-to-der}. We first note that $\cal D_Y^{(\nu)}\Mod^{\op{Fr}^*\rightarrow\op{id}}$ is right complete with respect to its $t$-structure (Lemma \ref{lem-right-complete}). Let $I, \cal M \in \cal D_{F,Y}^{(\nu)}\Mod^{\heartsuit}$, with $I$ being an injective object. By Lemma \ref{lem-fixed-point-hom}, the Hom-complex $R\Hom_{\cal D_Y^{(\nu)}\Mod^{\op{Fr}^*\rightarrow\op{id}}}(\cal M, I)$ identifies with the (homotopy) fiber of:
$$
\beta : R\Hom_{\cal D_Y^{(\nu)}\Mod}(\cal M, I) \rightarrow R\Hom_{\cal D_Y^{(\nu)}\Mod}(\op{Fr}_Y^*\cal M, I),\quad f\leadsto \psi_I\circ \op{Fr}_Y^*(f) - f\circ\psi_{\cal M}.
$$
We must show that this fiber has zero cohomology groups in degree $i\ge 1$.

\subsubsection{}
From the discussion in \S\ref{sec-crystal-adjunct}, we know that $I$ is also injective as an object in $\cal D_Y^{(\nu)}\Mod$. Hence we may replace $\beta$ by a map between \emph{underived} Hom-spaces:
$$
\beta : \Hom_{\cal D_Y^{(\nu)}\Mod^{\heartsuit}}(\cal M, I) \rightarrow \Hom_{\cal D_Y^{(\nu)}\Mod^{\heartsuit}}(\op{Fr}_Y^*\cal M, I).
$$
We now observe that the following diagram commutes:
$$
\xysmall{
	\Hom_{\cal D_Y^{(\nu)}\Mod^{\heartsuit}}(\cal M, I) \ar[r]^{\beta}\ar[d]^{\cong} & \Hom_{\cal D_Y^{(\nu)}\Mod^{\heartsuit}}(\op{Fr}_Y^*\cal M, I) \ar[d]^{\cong} \\
	\Hom_{\cal D_{F,Y}^{(\nu)}\Mod^{\heartsuit}}(\cal D_{F,Y}^{(\nu)}\underset{\cal D_Y^{(\nu)}}{\otimes}\op{Fr}_Y^*\cal M, I) \ar[r]^-{\alpha^*} \ar[r] & \Hom_{\cal D_{F,Y}^{(\nu)}\Mod^{\heartsuit}}(\cal D_{F,Y}^{(\nu)}\underset{\cal D_Y^{(\nu)}}{\otimes}\cal M, I)
}
$$
Here, the vertical isomorphisms come from the adjuction \eqref{eq-ind-crys}. The bottom map is precomposition with $\alpha$ in the two-step resolution \eqref{eq-two-step-filt}. This calculation shows that the fiber of $\beta$ is isomorphic to the abelian group $\Hom_{\cal D_{F,Y}^{(\nu)}\Mod^{\heartsuit}}(\cal M, I)$, as expected.
\qed(Proposition \ref{prop-ab-to-der})

\medskip

\section{The $p$-torsion Artin reciprocity functor}

In this section, we construct $p$-torsion Artin reciprocity as a contravariant functor from $\op{Coh}(\widetilde{\op{Jac}}{}_n^{\natural})^{\op{id}}_{\circ}$ to the category of $W_n(\mathbb F_q)$-sheaves on the $\op{Jac}_n$. Then we verify that this functor categorifies the usual Artian reciprocity map for $p$-torsion characters trivial on $x_0$. In fact, our construction applies more generally to any abelian scheme $A$ over $W_n(\mathbb F_q)$ and the Verschiebung--fixed locus of the universal additive extension of $A^{\vee}$.

\smallskip

Along the way, we will discuss how the Vershiebung--fixed locus $\widetilde{\op{Jac}}{}_n^{\natural}$ is related to characters of the \'etale fundamental group.

\subsection{The Jacobian scheme}
\label{sec-jac}

\subsubsection{} From now on, \emph{we will assume $k=\mathbb F_q$} for simplicity. The ring of length--$n$ Witt vectors will be denoted by $W_n$. Let $X$ be a smooth, proper, geometrically connected curve over $\mathbb F_q$. We assume that it contains an $\mathbb F_q$-rational point $x_0$. We fix a smooth lift $X_n$ over $W_n$; it always exists since the obstruction to such lifts lies in a degree--$2$ cohomology group. The smoothness implies that $x_0$ can be lifted to a $W_n$-rational point of $X_n$.

\subsubsection{} The Jacobian $\op{Jac}_n$ parametrizes degree-$0$ line bundles on $X_n$ together with a rigidification at $x_0$. It is representable by a smooth abelian scheme over $W_n$ (see, for example, \cite[Proposition 9.5.19]{Kl05}). There is a canonical self-duality of $\op{Jac}_n$, as witnessed by the following Poincar\'e line bundle $\cal P$ on $\op{Jac}_n \underset{W_n}{\times} \op{Jac}_n$:
\begin{align*}
\cal L_1, \cal L_2 \leadsto \det R\Gamma(X_n, \cal L_1\otimes \cal L_2) \otimes & \det R\Gamma(X_n, \cal L_1)^{\otimes -1} \\
& \otimes \det R\Gamma(X_n, \cal L_2)^{\otimes -1} \otimes \det R\Gamma(X_n, \cal O).
\end{align*}
The notation $R\Gamma(X_n, -)$ is a short-hand for pushfoward along $X_n \underset{W_n}{\times} S \rightarrow S$, where $S$ is a test $W_n$-scheme. The determinant is well-defined since $R\Gamma(X_n,-)$, applied to any coherent sheaf flat over $S$, yields a perfect complex.

\subsubsection{} To each abelian scheme $A$ over $W_n$, we attach an additive extension of its dual:
$$
0 \rightarrow \op H^0(\Omega_{A/W_n}) \rightarrow \widetilde A^{\vee} \xrightarrow{\pi} A^{\vee} \rightarrow 0.
$$
Explicitly, $\widetilde A^{\vee}$ classifies a multiplicative line bundle on $A$ equipped with an integrable connection. In other words, an $S$-point of $\widetilde A^{\vee}$ is a multiplicative line bundle on $S\underset{W_n}{\times} A$ whose $\cal O$-module structure enhances to that of $\cal D^{(0)}_{S\times A/S}$. We note that this datum is equivalent to a multiplicative pair $(\cal L, \nabla)$, where $\cal L$ is a line bundle on $S\underset{W_n}{\times} A$ and  $\nabla$ is an integrable connection on $\cal L$ along $A$ (\cite[Lemme 2.1.1]{La96}). One knows that $\widetilde A^{\vee}$ is representable by a smooth scheme of relative dimension $2g$ over $W_n$, and is in fact the universal additive extension of $A^{\vee}$ \cite[\S1]{MM06}. Pulling back along the Frobenius on $A$ (see \S\ref{sec-frob-pullback}) defines an endomorphism:
\begin{equation}
\label{eq-ver-defn}
\op{Ver} : \widetilde A^{\vee} \rightarrow \widetilde A^{\vee},
\end{equation}
which we will call the \emph{Verschiebung} endormophism.

\begin{rem}
Suppose $n=1$. Then \eqref{eq-ver-defn} lifts the usual Verschiebung endomorphism of the abelian variety $A^{\vee}$ over $\mathbb F_q$.
\end{rem}

\subsubsection{} We let $\widetilde A^{\vee, \natural}$ denote the Verschiebung-fixed point locus of $\widetilde A^{\vee}$. In other words, it is the \emph{a priori} derived closed subscheme given by the fiber product:
\begin{equation}
\label{eq-ver-fixed-locus}
\xymatrix@R=1.5em{
	\widetilde A^{\vee, \natural} \ar[r]^{\iota}\ar[d] & \widetilde A^{\vee} \ar[d]^{\Delta} \\
	\widetilde A^{\vee} \ar[r]^-{(\op{Ver},\op{id})} & \widetilde A^{\vee} \underset{W_n}{\times} \widetilde A^{\vee}
}
\end{equation}

\begin{lem}
\label{lem-univ-reg}
The map $\iota : \widetilde A^{\vee, \natural} \hookrightarrow \widetilde A^{\vee}$ is a regular immersion of classical schemes.
\end{lem}
\begin{proof}
By smoothness of $\widetilde A^{\vee}$ over $W_n$, we see that both embeddings of $\widetilde A^{\vee}$ in $\widetilde A^{\vee} \underset{W_n}{\times} \widetilde A^{\vee}$ in \eqref{eq-ver-fixed-locus} are regular immersions of half-dimensional subschemes; indeed, these are both sections of the projection map $\op{pr}_2 : \widetilde A^{\vee} \underset{W_n}{\times} \widetilde A^{\vee} \rightarrow \widetilde A^{\vee}$ which is smooth (c.f. \cite[067R]{Stacks}). Recall that given a Noetherian local ring $(R, \fr m)$ of dimension $d$ and elements $f_1,\cdots,f_d\in\fr m$, the following statements are equivalent:
\begin{enumerate}[(a)]
	\item the Koszul complex associated to $f_1, \cdots, f_d$ has cohomology only in degree $0$ (i.e., $f_1,\cdots,f_d$ is a regular sequence);
	\smallskip
	\item the Krull dimension of $R/(f_1,\cdots, f_d)$ vanishes.
\end{enumerate}
Using this fact, we see that the Lemma will follow from checking the Krull dimension of (the underlying classical scheme) of $\widetilde A^{\vee, \natural}$ vanishes. This statement can in turn be verified after base change to the special fiber over $\Spec(\mathbb F_q)$, so \emph{we may assume $n=1$}. In this case, we have the usual Verschiebung endomorphism on $A^{\vee}$, given by the Frobenius-pullback of line bundles on $A$. There is a commutative diagram:
$$
\xysmall{
	0 \ar[r] & \op H^0(\Omega_{A/\mathbb F_q}) \ar[r]\ar[d]^{\op{Fr}^* = 0} & \widetilde A^{\vee} \ar[r]^{\pi} \ar[d]^{\op{Ver}} & A^{\vee} \ar[r]\ar[d]^{\op{Ver}} & 0 \\
	0 \ar[r] & \op H^0(\Omega_{A/\mathbb F_q}) \ar[r] & \widetilde A^{\vee} \ar[r]^{\pi} & A^{\vee} \ar[r] & 0
}
$$
Consequently, $\pi$ identifies $\widetilde A^{\vee, \natural}$ with the Verschiebung-fixed point scheme $A^{\vee, \natural}$ of $A^{\vee}$. We now show that the latter is zero-dimensional. Note that $A^{\vee, \natural}$ is the fiber of the map of abelian varieties:
$$
0 \rightarrow A^{\vee, \natural} \rightarrow A^{\vee} \xrightarrow{\op{Ver} - \op{id}} A^{\vee}.
$$
Now, the endomorphism $\op{Ver} - \op{id}$ is an isogeny since it is dual to the Lang isogeny $\op{Fr} - \op{id} : A\rightarrow A$, so we are done.
\end{proof}

\begin{rem}
In particular, the proof shows that $\widetilde A^{\vee, \natural}$ is a finite $W_n$-scheme whose $\mathbb F_q$-points are identified with those of the Verschiebung fixed locus of the special fiber $A_1^{\vee}$ of $A^{\vee}$.
\end{rem}

\subsubsection{} Let $\widetilde{\op{Jac}}_n$ denote the abelian scheme classifying degree-$0$, rigidified line bundles on $X_n$ together with an integrable connection (relative to $W_n$). Then the auto-duality of $\op{Jac}_n$ enhances to an isomorphism of $\widetilde{\op{Jac}}_n$ with the universal additive extension of the dual of $\op{Jac}_n$. In other words, we have a commutative diagram:
$$
\xysmall{
	\widetilde{\op{Jac}}_n \ar[r]^-{\sim}\ar[d] &  \widetilde{\op{Jac}}{}_n^{\vee} \ar[d] \\
	\op{Jac}_n \ar[r]^-{\sim} & \op{Jac}_n^{\vee}
}
$$
Furthermore, the Verschiebung endomorphism on $\widetilde{\op{Jac}}{}_n^{\vee}$ passes to the endomorphism:
$$
\op{Ver} : \widetilde{\op{Jac}}_n \rightarrow \widetilde{\op{Jac}}_n,\quad \cal L \leadsto \op{Fr}_{X_n}^*\cal L.
$$

\subsubsection{}
\label{sec-verschiebung-fixed-point-locus}
We write $\widetilde{\op{Jac}}{}_n^{\natural}$ for the Verschiebung-fixed point locus of $\widetilde{\op{Jac}}_n$. Explicitly, an $S$-point of $\widetilde{\op{Jac}}{}_n^{\natural}$ is a degree-$0$, rigidified line bundle $\cal L$ on $X_n$ with an integrable connection, as well as an identification $\op{Fr}_{X_n}^*\cal L \xrightarrow{\sim} \cal L$ of $\cal D_{S\times X_n/S}^{(0)}$-modules compatible with the rigidification at $x_0$. Lemma \ref{lem-univ-reg} shows that $\widetilde{\op{Jac}}{}_n^{\natural}$ is a disjoint union of connected zero-dimensional schemes $\widetilde{\op{Jac}}{}_{n, \sigma}^{\natural}$, where $\sigma$ ranges through the $\mathbb F_q$-points of $\widetilde{\op{Jac}}{}_n^{\natural}$.

\subsection{Galois deformations}
\label{sec-galois-def}

\subsubsection{} In this subsection, we relate the Verschiebung-fixed point locus to Galois deformations over $W_n$. This relation is not needed for the construction of the Artin reciprocity functor.

\subsubsection{} Fix a geometric point $\overline x_0$ of the curve $X$ \emph{lying over} $x_0$. We use $\overline x_0$ to define the \'etale fundamental groups of $\overline X$ and $X$. There is a short exact sequence:
$$
1 \rightarrow \pi_1(\overline X) \rightarrow \pi_1(X) \rightarrow \pi_1(\Spec(\mathbb F_q)) \rightarrow 0,
$$
where the second map has a section defined by $x_0$. We will be concerned with continuous characters $\sigma : \pi_1(X) \rightarrow \mathbb F_q^{\times}$ which are \emph{trivial on $\pi_1(x_0)$}. We note that such objects are equivalently rank--$1$ \'etale $\mathbb F_q$-sheaves on $X$ trivialized at $x_0$. By Katz's equivalence, they are in turn $\mathbb F_q$-points of $\widetilde{\op{Jac}}{}^{\natural}_1$ (or equivalently $\widetilde{\op{Jac}}{}_n^{\natural}$).

\begin{rem}
We know \emph{a fortiori} that $\sigma$ is equivalent to a character of the geometric fundamental group $\pi_1(\overline X)$, but viewing it as such is less natural for our purpose.
\end{rem}

\subsubsection{} Let $\sigma : \pi_1(X) \rightarrow \mathbb F_q^{\times}$ be a continuous character which is trivial on $\pi_1(x_0)$. We define a functor $\op{Def}^{\sigma}_n$ on the category of Artinian, local $W_n$-algebras with residue field $\mathbb F_q$ as follows. For $S=\Spec(A)$ where $A$ is such a $W_n$-algebra, we let $\op{Def}_n^{\sigma}(A)$ be the set of continuous characters:
$$
\pi_1(X) \rightarrow A^{\times}, \quad\text{trivial on }\pi_1(x_0),
$$
such that its reduction along $A^{\times} \rightarrow \mathbb F_q^{\times}$ identifies with $\sigma$. Since $A$ is finite, the continuity requirement is equivalent to factoring through a finite quotient.

\subsubsection{} We now use regard $\sigma$ as an $\mathbb F_q$-point of $\widetilde{\op{Jac}}{}_n^{\natural}$.

\begin{lem}
\label{lem-galois-def}
The functor $\op{Def}_n^{\sigma}$ is represented by the \emph{pointed} $W_n$-scheme $\widetilde{\op{Jac}}{}_{n, \sigma}^{\natural}$.
\end{lem}
\begin{proof}
Given a Artinian, local $W_n$-algebra $A$ with residue field $\mathbb F_q$, write $S=\Spec(A)$. An element of $\op{Def}_n^{\sigma}(A)$ is equivalent to a rigidified \'etale $A$-local system on $X_n$ of rank $1$, whose induced $\mathbb F_q$-local system identifies with $\sigma$. By Lemma \ref{lem-katz-equiv}, this datum is equivalent to a line bundle $\cal L$ over $S \underset{W_n}{\times} X_n$ together with a connection along $X_n$ and an isomorphism $\cal L\xrightarrow{\sim} (\op{id}_S \times \op{Fr}_X)^*\cal L$, whose restriction to $x_0$ is rigidified and whose reduction to $\mathbb F_q$ identifies with $\sigma$. The latter is precisely that of an $S$-point of $\widetilde{\op{Jac}}{}_n^{\natural}$ whose closed point is $\sigma$.
\end{proof}

\subsubsection{} In particular, we obtain a fully faithful functor:
\begin{equation}
\label{eq-galois-def}
\QCoh(\op{Def}_n^{\sigma}) \xrightarrow{\sim} \QCoh(\widetilde{\op{Jac}}{}_{n, \sigma}^{\natural}) \hookrightarrow \QCoh(\widetilde{\op{Jac}}{}_n^{\natural})
\end{equation}

\begin{rem}
The argument of Lemma \ref{lem-galois-def} works more generally for representations of higher rank. Namely, for a rank--$r$ representation $\sigma$, the functor of Galois deformations $\op{Def}_n^{\sigma}$ over $W_n$ is represented by the formal neighborhood of the $W_n$-stack $\op{LocSys}^{\natural}_{r, n}$ at $\sigma$, regarded as a pointed $W_n$-stack. However, for $r\ge 2$, $\op{LocSys}^{\natural}_{r, n}$ may no longer be zero-dimensional, so it contains more information than $\op{Def}_n^{\sigma}$ (c.f.~Y.~Laszlo \cite{La01}).
\end{rem}

\subsection{Fourier--Mukai--Laumon transform}

\subsubsection{} Let $A$ be an abelian scheme over $W_n$. We recall that the Fourier--Mukai transform defines an equivalence of $W_n$-linear $\infty$-categories:
\begin{equation}
\label{eq-fm-transform}
\Phi : \QCoh(A^{\vee}) \xrightarrow{\sim} \QCoh(A), \quad \cal F\leadsto Rq_*(p^*\cal F \overset{L}{\otimes} \cal P),
\end{equation}
where $\cal P$ is the universal line bundle on $A^{\vee} \underset{W_n}{\times} A$ and $p$, $q$ are the two projections.

\subsubsection{} The Fourier--Mukai--Laumon transform (c.f.~\cite[Th\'eor\`eme 3.2.1]{La96}) enhances \eqref{eq-fm-transform} into an equivalence between $\QCoh(\widetilde A^{\vee})$ and $\cal D^{(0)}\Mod(A)$. More precisely, we write $\widetilde{\cal P}$ for the universal line bundle on $\widetilde A^{\vee} \underset{W_n}{\times} A$ equipped with an integrable connection along $A$ and $\widetilde p ,\widetilde q$ for the two projections.

\begin{lem}
The functor:
\begin{equation}
\label{eq-fml-transform}
\widetilde{\Phi} : \QCoh(\widetilde A^{\vee}) \rightarrow \cal D^{(0)}\Mod(A) ,\quad \cal F \leadsto R\widetilde q_*(\widetilde p^*\cal F \overset{L}{\otimes} \widetilde{\cal P})
\end{equation}
defines an equivalence of $W_n$-linear $\infty$-categories, making the following diagram commutes:
\begin{equation}
\label{eq-fml-transform-oblv}
\xysmall{
	\QCoh(\widetilde A^{\vee}) \ar[r]^-{\widetilde{\Phi}}\ar[d]^{\pi_*} & \cal D^{(0)}\Mod(A) \ar[d]^{\op{oblv}_{\cal D^{(0)}}} \\
	\QCoh(A^{\vee}) \ar[r]^-{\Phi} & \QCoh(A)
}
\end{equation}
\end{lem}

\noindent
The result in \emph{loc.cit.~}is proved when the base is in characteristic zero, although this assumption is superfluous. We indicate the necessary modification of Laumon's argument which proves the result for an abelian scheme $A$ over any locally Noetherian base scheme $S$. (We then apply the result to $S=W_n$.)

\begin{proof}[Sketch of proof]
The Lemma follows formally from two canonical isomorphisms:
\begin{enumerate}[(a)]
	\item The $\cal D^{(0)}$-module pushforward along $\widetilde p$ of $\widetilde{\cal P}$:
	$$
	\widetilde p_{*,\dR}(\widetilde{\cal P}) \xrightarrow{\sim} R\widetilde p_* (\widetilde{\cal P} \xrightarrow{\nabla} \widetilde{\cal P}\otimes\Omega_{A}^1\xrightarrow{\nabla} \cdots)
	$$
	ought to be identified with $\widetilde{\epsilon}_*\cal O_S[-g]$, where $\widetilde{\epsilon}: S\rightarrow\widetilde A^{\vee}$ is the identity;
	
	\medskip
	
	\item The pushforward along $\widetilde q$ of $\widetilde{\cal P}$ ought to be identified with $\epsilon_{*,\dR}(\cal O_S)$, where $\epsilon_{*,\dR}$ is the $\cal D^{(0)}$-module pushforward along the identity $\epsilon : S\rightarrow A$.\footnote{We remark that $\epsilon_{*,\dR}(\cal O_S)$, considered as a left $\cal D^{(0)}$-module, lives in cohomological degree $g$.}
\end{enumerate}
\noindent
The argument in \emph{loc.cit.}~works verbatim for the first identification, but it appeals to Kashiwara's lemma for the second identification, which does \emph{not} hold outside characteristic zero. We will provide this isomorphism in a more direct manner. First, there is a canonically defined morphism of $\cal D_{A/S}^{(0)}$-modules:
\begin{equation}
\label{eq-laumon-isom}
\epsilon_{*,\dR} (\cal O_S) \rightarrow R\widetilde q_*\widetilde{\cal P},
\end{equation}
obtained from adjunction by a map of $\cal O_S$-modules $\cal O_S \rightarrow L\epsilon^* R\widetilde q_*\widetilde{\cal P}$, which in turn comes from base change. Next, we observe that $R\widetilde q_*\widetilde{\cal P}$ identifies with $Rq_*(\cal O_{\widetilde A^{\vee}} \otimes \cal P)$. By \S2.3 of \emph{loc.cit.}, $\cal O_{\widetilde A^{\vee}}$ admits a filtration $F^{\le i}\cal O_{\widetilde A^{\vee}}$ as an $\cal O_{A^{\vee}}$-module, such that:
$$
F^{\le i}\cal O_{\widetilde A^{\vee}}/F^{\le i-1}\cal O_{\widetilde A^{\vee}} \xrightarrow{\sim}\op{Sym}^i(\epsilon^*\cal T_{A/S})\underset{\cal O_S}{\otimes}\cal O_{A^{\vee}}.
$$
Since $Rq_*(\cal P)$ is canonically isomorphic to $\epsilon_*\epsilon^*\omega_{A/S}^{\otimes -1}[-g]$ (Lemme 1.2.5 of \emph{loc.cit.}), we see that $R\widetilde q_*\widetilde{\cal P}$ is concentrated in cohomological degree $g$ and inherits a filtration by $\cal O_A$-submodules. On the other hand, $\epsilon_{*,\dR}(\cal O_S)$ admits a filtration as an $\cal O_A$-module by order of differential operators in $\cal D^{(0)}_{A/S}$, whose associated graded pieces are given by:
$$
F^{\le i}\epsilon_{*,\dR}(\cal O_S)/F^{\le i-1}\epsilon_{*,\dR}(\cal O_S) \xrightarrow{\sim} \op{Sym}^i(\epsilon^*\cal T_{A/S}) \underset{\cal O_S}{\otimes} \epsilon_*\epsilon^*\omega_{A/S}^{\otimes -1}[-g];
$$
here, the factor $\omega_{A/S}^{\otimes -1}[-g]$ arises from passing the $\cal D_{A/S}^{(0)}$-action from right to left. The morphism \eqref{eq-laumon-isom} is compatible with the filtrations, i.e., the following diagram commutes:
$$
\xysmall{
	F^{\le i}\epsilon_{*,\dR}(\cal O_S) \ar[r]\ar[d] & F^{\le i} R\widetilde q_*\widetilde{\cal P} \ar[d] \\
	\epsilon_{*,\dR}(\cal O_S) \ar[r] & R\widetilde q_*\widetilde{\cal P}
}
$$
Indeed, this is a consequence of the two facts below:
\begin{enumerate}[(a)]
	\item the adjoint morphism $\cal O_S \rightarrow L\epsilon^* Rq_*(\cal O_{\widetilde A^{\vee}}\otimes\cal P)$ of \eqref{eq-laumon-isom} factors through
	$$
	L\epsilon^* Rq_*(F^{\le 0}\cal O_{\widetilde A^{\vee}}\otimes \cal P) \cong L\epsilon^* Rq_*\cal P;
	$$
	\item the filtration on $R\widetilde q_*\widetilde{\cal P}$ is compatible with $\cal D_{A/S}^{(0)}$-action, in the sense that:
	$$
	F^{\le i}\cal D_{A/S}^{(0)} \cdot F^{\le j} R\widetilde q_*\widetilde{\cal P} \subset F^{\le i+j} R\widetilde q_*\widetilde{\cal P}.
	$$
\end{enumerate}
\noindent
Thus it remains to show that \eqref{eq-laumon-isom} induces an isomorphism on each graded piece. On the $i$th associated graded piece, it induces the following map:
$$
\op{Sym}^i(\epsilon^*\cal T_{A/S}) \underset{\cal O_S}{\otimes} \epsilon_*\epsilon^*(\omega_{A/S}^{\otimes -1})[-g] \xrightarrow{\op{id}\otimes \alpha} \op{Sym}^i(\epsilon^*\cal T_{A/S}) \underset{\cal O_S}{\otimes} Rq_*\cal P,
$$
where $\alpha$ is the canonical isomorphism $\epsilon_*\epsilon^*(\omega_{A/S}^{\otimes -1})[-g] \xrightarrow{\sim} Rq_*\cal P$.
\end{proof}

\subsubsection{} We show that the Fourier--Mukai--Laumon transform satisfies the Frobenius compatibility alluded to in the introduction. We emphasize that there is no analogous diagram for the usual Fourier--Mukai transform $\Phi$ beyond the special fiber.

\begin{lem}
The following diagram commutes:
\begin{equation}
\label{eq-frob-ver}
\xysmall{
	\QCoh(\widetilde A^{\vee}) \ar[r]^-{\widetilde{\Phi}}\ar[d]^{\op{Ver}_*} & \cal D^{(0)}\Mod(A) \ar[d]^{\op{Fr}^*_A} \\
	\QCoh(\widetilde A^{\vee})  \ar[r]^-{\widetilde{\Phi}} & \cal D^{(0)}\Mod(A)
}
\end{equation}
\end{lem}
\begin{proof}
By definition of the Verschiebung endomorphism, we have an isomorphism:
\begin{equation}
\label{eq-univ-pullback}
(\op{Ver} \times \op{id})^*\widetilde{\cal P} \xrightarrow{\sim} \op{Fr}_A^*\widetilde{\cal P}
\end{equation}
of $\cal D^{(0)}_{\widetilde A^{\vee} \times A/\widetilde A^{\vee}}$-modules. Now, consider the following commutative diagram:
$$
\xymatrix@R=1em@C=2em{
	& \widetilde A^{\vee} \times A \ar[dl]_{\widetilde p}\ar[dr]^{\widetilde q} &  \\
	\widetilde A^{\vee} & & A \ar@{.>}[dd]^{\op{Fr}_A} \\
	& \widetilde A^{\vee} \times A \ar[dl]_{\widetilde p} \ar[ur]_{\widetilde q} \ar[uu]^{\op{Ver}\times\op{id}} \ar@{.>}[dd]^{\op{id}\times\op{Fr}_A} & \\
	\widetilde A^{\vee} \ar[uu]^{\op{Ver}} & & A \\
	& \widetilde A^{\vee} \times A \ar[ul]^{\widetilde p} \ar[ur]_{\widetilde q} &
}
$$
where the dotted arrows indicate that $\op{Fr}_A$ is only well defined affine locally on $A$, but suffices to define the pullback of $\cal D^{(0)}$-modules. Using the projection formula, both circuits in \eqref{eq-frob-ver} can be identified with pulling back along $\widetilde p$, tensoring with the $\cal D_{\widetilde A^{\vee}\times A/\widetilde A^{\vee}}^{(0)}$-module \eqref{eq-univ-pullback}, and then pushing forward along $\widetilde q$.
\end{proof}

\subsubsection{} We now let $\QCoh(\widetilde A^{\vee})^{\op{Ver}_*}$ and $\cal D^{(0)}\Mod(A)^{\op{Fr}^*}$ be the categorical fixed points (c.f.~\S\ref{sec-cat-fix}) of these $W_n$-linear $\infty$-categories under the endomorphisms $\op{Ver}_*$ and $\op{Fr}_A^*$. It follows from the Lemma that we have a $W_n$-linear equivalence:
$$
	\QCoh(\widetilde A^{\vee})^{\op{Ver}_*} \xrightarrow{\sim} \cal D^{(0)}\Mod(A)^{\op{Fr}^*}.
$$

\subsubsection{} Finally, setting $A:=\op{Jac}_n$ and using the auto-duality explained in \S\ref{sec-jac}, we obtain an equivalence:
\begin{equation}
\label{eq-final-transform}
\QCoh(\widetilde{\op{Jac}}_n)^{\op{Ver}_*} \xrightarrow{\sim} \cal D^{(0)}\Mod(\op{Jac}_n)^{\op{Fr}^*}
\end{equation}
which makes the following diagram commute:
$$
\xysmall{
	\QCoh(\widetilde{\op{Jac}}_n)^{\op{Ver}_*} \ar[r]^-{\eqref{eq-final-transform}}\ar[d]^{\op{oblv}} & \cal D^{(0)}\Mod(\op{Jac}_n)^{\op{Fr}^*} \ar[d]^{\op{oblv}} \\
	\QCoh(\widetilde{\op{Jac}}_n) \ar[r]^-{\widetilde{\Phi}} & \cal D^{(0)}\Mod(\op{Jac}_n)
}
$$

\subsection{The $p$-torsion Artin reciprocity functor}
\label{sec-artin-recip}

\subsubsection{} Consider the $\infty$-category consisting of an object $\cal F\in\Coh(\widetilde{\op{Jac}}{}_n^{\natural})$ together with an automorphism $\alpha$. It can be regarded as fixed points under the identity endomorphism on $\Coh(\widetilde{\op{Jac}}{}_n^{\natural})$, denoted by $\op{Coh}(\widetilde{\op{Jac}}{}_n^{\natural})^{\op{id}}$. Note that there is a functor:
\begin{equation}
\label{eq-geom-to-cat-funct}
\Coh(\widetilde{\op{Jac}}{}_n^{\natural})^{\op{id}} \rightarrow \Coh(\widetilde{\op{Jac}}_n)^{\op{Ver}_*},
\end{equation}
defined by sending $(\cal F, \alpha)$ to the coherent sheaf $\iota_*\cal F$ together with the isomorphism:
$$
\iota_*\cal F \xrightarrow{\iota_*\alpha} \iota_*\cal F \xrightarrow{\sim} \op{Ver}_*\iota_*\cal F.
$$
Here, $\iota$ denotes the closed immersion of $\widetilde{\op{Jac}}{}_n^{\natural}$ in $\widetilde{\op{Jac}}_n$.

\subsubsection{} To summarize, we have constructed a commutative diagram:
\begin{equation}
\label{eq-transform-summary}
\xysmall{
\Coh(\widetilde{\op{Jac}}{}_n^{\natural})^{\op{id}} \ar[r]^-{\eqref{eq-geom-to-cat-funct}} &  \Coh(\widetilde{\op{Jac}}_n)^{\op{Ver}_*} \ar[r]^-{\eqref{eq-final-transform}} \ar[d]^{\op{oblv}} & \cal D^{(0)}\Mod(\op{Jac}_n)^{\op{Fr}^*} \ar[d]^{\op{oblv}} \\
 & \Coh(\widetilde{\op{Jac}}_n) \ar[r]^-{\widetilde{\Phi}} \ar[d]^{\pi_*} & \cal D^{(0)}\Mod(\op{Jac}_n) \ar[d]^{\op{oblv}_{\cal D^{(0)}}} \\
 & \Coh(\op{Jac}_n) \ar[r]^-{\Phi} & \QCoh(\op{Jac}_n)
}
\end{equation}

\subsubsection{} We write $\Coh(\widetilde{\op{Jac}}{}_n^{\natural})^{\op{id}}_{\circ}$ for the full subcategory of $\Coh(\widetilde{\op{Jac}}{}_n^{\natural})^{\op{id}}$ consisting of $(\cal F, \alpha)$ where $\cal F$ has finite Tor dimension over the base ring $W_n$.

\begin{lem}
Under the upper horizontal arrow of \eqref{eq-transform-summary}, the essential image of $\Coh(\widetilde{\op{Jac}}{}_n^{\natural})^{\op{id}}_{\circ}$ lies inside $(\cal D^{(0)}\Mod(\op{Jac}_n)^{\op{Fr}^*})_{\op{lfg-u}, \circ}$.
\end{lem}

\noindent
We refer the reader to \S\ref{sec-f-crys-fixed-points} for the notation.

\begin{proof}
Take $(\cal F, \alpha) \in \Coh(\widetilde{\op{Jac}}{}_n^{\natural})^{\op{id}}_{\circ}$ whose image in $\cal D^{(0)}\Mod(\op{Jac}_n)^{\op{Fr}^*}$ is denoted by $\cal M$. The underlying object in $\QCoh(\op{Jac}_n)$ of $\cal M$ is \emph{coherent} by the commutativity of the diagram \eqref{eq-transform-summary}. This shows that it lies inside $(\cal D^{(0)}\Mod(\op{Jac}_n)^{\op{Fr}^*})_{\op{lfg-u}}$. It remains to check that $\cal M$ has finite Tor dimension over the structure sheaf of $\op{Jac}_n$. We note:
\begin{enumerate}[(a)]
	\item $\cal M$ has finite Tor dimension over $W_n$; indeed, this is because all functors involved are $W_n$-linear and $\cal F$ has finite Tor dimension over $W_n$. In particular, this implies that the (derived) restriction $\cal M|_{\op{Jac}_1}$ to the special fiber is a bounded complex of coherent sheaves.
	\item The restriction $\cal M|_{\op{Jac}_1}$ is in fact perfect; indeed, this is because each of its cohomology groups is an $\cal O$-coherent $F$-crystal, and we may apply \cite[Proposition 6.9.3]{EK04} to conclude that it is locally free.
\end{enumerate}
We let $\cal P^{\bullet} \in \Coh(\op{Jac}_1)^{\ge a, \le b}$ be the perfect complex representing $\cal M\big|_{\op{Jac}_1}$. For every closed point $y\in\op{Jac}_n$ (necessarily factoring through $\op{Jac}_1$), we have
$$
Li_y^*\cal M \xrightarrow{\sim} \cal P^{\bullet} \otimes k_y \in \op{Vect}^{\ge a, \le b}.
$$
This implies that $\cal M$ has Tor amplitude $[a,b]$ over $\op{Jac}_n$ by the following general fact.
\end{proof}

\begin{lem}
Suppose $Y$ is a Noetherian scheme, and $\cal M\in \Coh(Y)$ has the property that $Li_y^*\cal M \in\op{Vect}^{\ge a, \le b}$ for all closed points $y\in Y$. Then $\cal M$ has Tor amplitude $[a,b]$.
\end{lem}
\begin{proof}
We immediately reduce to the case where $Y$ is the spectrum of a local ring. By Nakayama Lemma, we may choose a free resolution $\cal P^{\bullet}$ of $\cal M$ where each differential $\delta$ reduces to zero modulo $\fr m_y$. Thus $Li_y^*\cal M \xrightarrow{\sim} \bigoplus_i \cal P^i \underset{\cal O_Y}{\otimes} k_y[-i]$. The hypothesis then implies that $\cal P^i = 0$ for $i\notin [a,b]$.
\end{proof}

\subsubsection{} We may thus use the composition of \eqref{eq-final-transform} and \eqref{eq-geom-to-cat-funct} to define the \emph{$p$-torsion Artin reciprocity functor}:
\begin{align*}
\mathbb L_n : \Coh(\widetilde{\op{Jac}}{}_n^{\natural} )^{\op{id}}_{\circ} \rightarrow & (\cal D^{(0)}\Mod(\op{Jac}_n)^{\op{Fr}^*})_{\op{lfg-u}, \circ} \\
& \xrightarrow{\sim} W_n(\mathbb F_q)\Shv_c(\op{Jac})^{\op{op}},
\end{align*}
where the second functor is the Riemann--Hilbert correspondence of Corollary \ref{cor-ekrh}.

\begin{rem}
For a fixed character $\sigma : \pi_1(X) \rightarrow \mathbb F_q^{\times}$ trivial on $\pi_1(x_0)$, one may also regard the $p$-torsion Artin reciprocity as a functor out of $\Coh(\op{Def}_n^{\sigma})_{\circ}^{\op{id}}$ by precomposing $\mathbb L_n$ with \eqref{eq-galois-def}:
$$
\mathbb L_n^{\sigma} : \op{Coh}(\op{Def}_n^{\sigma})_{\circ}^{\op{id}} \hookrightarrow \Coh(\widetilde{\op{Jac}}{}_n^{\natural} )^{\op{id}}_{\circ} \xrightarrow{\mathbb L_n} W_n(\mathbb F_q)\Shv_c(\op{Jac})^{\op{op}}.
$$
\end{rem}

\subsubsection{} Suppose now that $X$ is equipped with a lift to the formal scheme $W := \underset{n}{\op{colim}}\,W_n$. By construction, the following diagram commutes:
$$
\xysmall{
	 \Coh(\widetilde{\op{Jac}}{}_n^{\natural})^{\op{id}}_{\circ} \ar[r]^-{\mathbb L_n} & W_n\Shv_c(\op{Jac})^{\op{op}} \\
	 \Coh(\widetilde{\op{Jac}}{}_{n+1}^{\natural})^{\op{id}}_{\circ} \ar[r]^-{\mathbb L_{n+1}} \ar[u]_{\op{Res}} & W_{n+1}\Shv_c(\op{Jac})^{\op{op}} \ar[u]_{(-)\underset{W_{n+1}}{\otimes} W_n}
}
$$
Therefore, letting $\widetilde{\op{Jac}}{}^{\natural}_{\infty}$ be the formal scheme $\underset{n}{\op{colim}}\,\widetilde{\op{Jac}}{}_n^{\natural}$, we obtain a functor:
$$
\mathbb L_{\infty} : \Coh(\widetilde{\op{Jac}}{}^{\natural}_{\infty})^{\op{id}}_{\circ} \rightarrow \lim_nW_n\Shv_c(\op{Jac})^{\op{op}}
$$
to the derived $\infty$-category of $p$-adic constructible sheaves.

\subsection{Hecke eigen-property}

\subsubsection{} We recall that the Artin reciprocity map in geometric class field theory is a homomorphism of pro-finite groups: $\theta : \op{Pic}(\mathbb F_q) \rightarrow \pi_1(X)^{\op{ab}}$. It induces a map:
$$
\theta^* : \{W_n^{\times}\text{-characters of }\pi_1(X)\} \rightarrow \{\text{morphisms }\op{Pic}(\mathbb F_q) \rightarrow W_n^{\times}\}.
$$
Using $\mathbb L_n$, we will reconstruct the part of $\theta^*$ on $W_n^{\times}$-characters which are trivial on $\pi_1(x_0)$.

\subsubsection{} Given a character $\rho : \pi_1(X) \rightarrow W_n^{\times}$ trivial on $\pi_1(x_0)$, we will construct a character sheaf $\op{Aut}_{\rho} \in W_n\text{-Shv}_c(\op{Jac})$ equipped with a canonical isomorphism (i.e., the Hecke eigen-property):
\begin{equation}
\label{eq-hecke-isom}
	\op{add}^!\op{Aut}_{\rho} \xrightarrow{\sim} \rho \boxtimes \op{Aut}_{\rho} [1],
\end{equation}
where $\op{add}$ denotes the morphism:
$$
	\op{add} : X\times \op{Jac} \rightarrow \op{Jac},\quad (x, \cal L)\leadsto \cal L(x - x_0).
$$
Translation by $\cal O(x_0) \in \op{Pic}^1(\mathbb F_q)$ then produces a character sheaf on $\Pic$ (still denoted by $\op{Aut}_{\rho}$), and we will set $\theta^*\rho$ to be the trace of Frobenius on $\op{Aut}_{\rho}$ at each $\mathbb F_q$-point of $\op{Pic}$.

\subsubsection{}
By Katz's equivalence (Lemma \ref{lem-katz-equiv} for $A=W_n$), $\rho$ defines a line bundle $\cal L_{\rho}$ over $X_n$ with a connection, together with an isomorphism $\op{Fr}_{X_n}^*\cal L_{\rho} \xrightarrow{\sim} \cal L_{\rho}$ and a rigidification of these data at $x_0$. Write $\cal L_{\rho^{-1}}$ for its dual, so that $\op{Sol}(\cal L_{\rho^{-1}}) \xrightarrow{\sim} \rho[1]$ under the solution functor of Emerton--Kisin (see \S\ref{sec-katz-to-ek}). It gives rise to a $W_n$-point of $\widetilde{\op{Jac}}{}_n^{\natural}$:
$$
i_{\rho^{-1}} : \Spec(W_n) \rightarrow \widetilde{\op{Jac}}{}_n^{\natural}.
$$
When equipped with the \emph{identity} automorphism, $(i_{\rho^{-1}})_*\cal O$ can be regarded as an object in $\Coh(\widetilde{\op{Jac}}{}_n^{\natural})^{\op{id}}_{\circ}$. Define $\op{Aut}_{\rho}$ as its image under $\mathbb L_n$. In other words, it is the solution complex of the object:
$$
\cal E_{\rho^{-1}} \in \cal D^{(0)}\Mod(\op{Jac}_n)^{\op{Fr}^*}_{\op{lfg-u},\circ}
$$
attached to $(i_{\rho^{-1}})_*\cal O$.

\subsubsection{}
We now show that $\op{Aut}_{\rho}$ is a character sheaf and construct the isomorphism \eqref{eq-hecke-isom}. This latter will arise from an isomorphism:
\begin{equation}
\label{eq-hecke-isom-dmod}
	\op{add}^*_n (\cal E_{\rho^{-1}}) \xrightarrow{\sim} \cal L_{\rho^{-1}} \boxtimes \cal E_{\rho^{-1}}
\end{equation}
as $\cal D^{(0)}$-modules over $X_n\underset{W_n}{\times}\op{Jac}_n$ together with a commutative diagram:
\begin{equation}
\label{eq-hecke-isom-comp}
\xysmall{
	\op{add}_n^* (\cal E_{\rho^{-1}}) \ar[d]^{\op{add}_n^*\varphi_{\cal E_{\rho^{-1}}}}\ar[r] & \cal L_{\rho^{-1}} \boxtimes \cal E_{\rho^{-1}} \ar[d]^{\varphi_{\cal L_{\rho^{-1}}}\boxtimes \varphi_{\cal E_{\rho^{-1}}} } \\
	\op{Fr}^*_{\op{Jac}_n}\op{add}_n^* (\cal E_{\rho^{-1}}) \ar[r] & \op{Fr}_{X_n}^*L_{\rho^{-1}} \boxtimes \op{Fr}_{\op{Jac}_n}^*\cal E_{\rho^{-1}}
}
\end{equation}
Here, the maps $\varphi_{\cal E_{\rho^{-1}}}$ and $\varphi_{\cal L_{\rho^{-1}}}$ are the structural maps of the respective Frobenius $\cal D^{(0)}$-modules.

\begin{rem}
The operation that we call $*$-pullback on $\cal D^{(0)}$-modules passes to $!$-pullback of \'etale $W_n$-sheaves. This explains the notational difference between \eqref{eq-hecke-isom} and \eqref{eq-hecke-isom-dmod}.
\end{rem}

\subsubsection{} Indeed, the Fourier--Mukai--Laumon transfrom of $(i_{\rho^{-1}})_*\cal O$ is a multiplicative object in $\cal D^{(0)}\Mod(\op{Jac}_n)$ whose pullback along the Abel--Jacobi map:
$$
X_n \rightarrow \op{Jac}_n,\quad x\leadsto \cal O(x - x_0)
$$
identifies with $\cal L_{\rho^{-1}}$. One thus obtains \eqref{eq-hecke-isom-dmod} from the factorization of $\op{add}_n$ as:
$$
X_n \underset{W_n}{\times} \op{Jac}_n \xrightarrow{\op{AJ}\times\op{id}} \op{Jac}_n \underset{W_n}{\times} \op{Jac}_n \xrightarrow{m} \op{Jac}_n.
$$
The commutative diagram \eqref{eq-hecke-isom-comp} is a consequence of the functoriality of the construction.

\medskip

\section{Criterion of fully faithfulness}

In this section, we prove our main result concerning the behavior of the $p$-torsion Artin reciprocity functor $\mathbb L_n$. Namely, it is fully faithful if and only if the Hasse--Witt matrix of $X$ is nilpotent.

\subsection{Main result}

\subsubsection{} We remain in the setting $k=\mathbb F_q$. Let $X$ be a smooth, proper, geometrically connected curve over $\mathbb F_q$, equipped with a smooth lift $X_n$ to the ring $W_n$. In this setting, we have defined the reciprocity functor $\mathbb L_n$ (see \S\ref{sec-artin-recip}).

\subsubsection{} Our main result on the behavior of $\mathbb L_n$ is as follows.

\begin{thm}
\label{thm-nilpotence}
The following are equivalent:
\begin{enumerate}[(a)]
	\item The functor $\mathbb L_n$ is fully faithful;
	\item The endomorphism $\op{Fr}_X^*$ on $\op H^1(X, \cal O_X)$ is nilpotent.
\end{enumerate}
\end{thm}

\noindent
The Frobenius endomorphism on $\op H^1(X, \cal O_X)$ is known as the Hasse--Witt matrix of the curve $X$. For example, when $X$ is an elliptic curve, the endomorphism $\op{Fr}_X^*$ vanishes if and only if $X$ is supersingular.

\subsection{Reduction}

\subsubsection{} We first relate condition (b) to the geometry of the embedding of the Verschiebung-fixed locus $\widetilde{\op{Jac}}{}_n^{\natural} \hookrightarrow \widetilde{\op{Jac}}_n$ (see \S\ref{sec-verschiebung-fixed-point-locus}). Let $N_{\widetilde{\op{Jac}}{}_n^{\natural}/\widetilde{\op{Jac}}_n}$ denote the conormal sheaf of this closed immersion. Since the Verschiebung endomorphism of $\widetilde{\op{Jac}}_n$ induces the identity map on $\widetilde{\op{Jac}}{}_n^{\natural}$, it defines an endomorphism of $N_{\widetilde{\op{Jac}}{}_n^{\natural}/\widetilde{\op{Jac}}_n}$.

\begin{lem}
Condition (b) is equivalent to:
\begin{enumerate}[(b')]
	\item The endomorphism on $N_{\widetilde{\op{Jac}}{}_n^{\natural}/\widetilde{\op{Jac}}_n}$ induced by the Verschiebung is nilpotent.
\end{enumerate}
\end{lem}
\begin{proof}
The proof consists of two parts. In the first part, we prove that (b) is equivalent to the assertion that the endomorphism on the tangent bundle $T_{\widetilde{\op{Jac}}_n/W_n}|_{\widetilde{\op{Jac}}{}_n^{\natural}}$ induced from the Verschiebung is nilpotent. The latter endomorphism is nilpotent if and only if it is so on the special fiber $ T_{\widetilde{\op{Jac}}_1/\mathbb F_q}|_{\widetilde{\op{Jac}}{}_1^{\natural}}$. Using the group structure of $\widetilde{\op{Jac}}_1$ and the fact that $\widetilde{\op{Jac}}{}_1^{\natural}$ is zero-dimensional (c.f.~Lemma \ref{lem-univ-reg}), this condition translates to nilpotence of the Verschiebung on $ T_{\widetilde{\op{Jac}}_1/\mathbb F_q}|_e$, where $e$ is the identity element of $\widetilde{\op{Jac}}_1$. There is an exact sequence of $\mathbb F_q$-vector spaces:
$$
0 \rightarrow \op H^0(\Omega_{X/\mathbb F_q}) \rightarrow T_{\widetilde{\op{Jac}}_1/\mathbb F_q} |_e \xrightarrow{\pi} T_{\op{Jac}_1/\mathbb F_q} |_e \rightarrow 0.
$$
Note that the map induced by the Verschiebung is zero on $\op H^0(\Omega_{X/\mathbb F_q})$; on the other hand, we have a commutative diagram:
$$
\xysmall{
	T_{\op{Jac}_1/\mathbb F_q}|_e \ar[r]^-{\sim} \ar[d]^{\op{Ver}} & \op H^1(X, \cal O_X) \ar[d]^{\op{Fr}_X^*} \\
	T_{\op{Jac}_1/\mathbb F_q}|_e \ar[r]^-{\sim} & \op H^1(X, \cal O_X)
}
$$
This shows that the condition (b) is equivalent to the nilpotence of the Verschiebung endomorphism on $T_{\widetilde{\op{Jac}}_n/W_n}|_{\widetilde{\op{Jac}}{}_n^{\natural}}$. Next, we will construct an isomorphism between $N_{\widetilde{\op{Jac}}{}_n^{\natural}/\widetilde{\op{Jac}}_n}$ and the restriction of the cotangent bundle $\Omega_{\widetilde{\op{Jac}}_n/W_n}|_{\widetilde{\op{Jac}}{}_n^{\natural}}$ which intertwines the Verschiebung endomorphism. This will complete the proof of the Lemma.

For this purpose, we consider a more abstract setting: $S$ is a Noetherian base scheme, $X\rightarrow S$ a smooth scheme, $v : X\rightarrow X$ is an endomorphism, $X^v$ the derived fixed point subscheme:
\begin{equation}
\label{eq-derived-fixed-point-square}
\xysmall{
	X^v \ar[r]\ar[d] & X \ar[d]^{\Delta} \\
	X \ar[r]^-{(v,\op{id})} & X \underset{S}{\times} X 
}
\end{equation}
Suppose $X^v \rightarrow X$ is a regular closed immersion of classical schemes. (In our context, this hypothesis is verified in Lemma \ref{lem-univ-reg}.) Hence the cotangent complex $L_{X^v/X}$ is quasi-isomorphic to $N_{X^v/X}[1]$. On the other hand, since \eqref{eq-derived-fixed-point-square} is a derived fiber product, $L_{X^v/X}$ is identified with the restriction $L_{X/(X\underset{S}{\times}X)}|_{X^v}$ where the immersion $X\rightarrow X\underset{S}{\times} X$ is the diagonal map. The smoothness of $X\rightarrow S$ yields a canonical identification of $L_{X/(X\underset{S}{\times} X)}$ with $\Omega_{X/S}[1]$, so we find an isomorphism $N_{X^v/X} \xrightarrow{\sim} \Omega_{X/S}|_{X^v}$. It is straightforward to check that this isomorphism intertwines the endomorphisms induced from $v$.
\end{proof}

\subsubsection{}
Let us now inspect condition (a). By construction of the functor $\mathbb L_n$, it is fully faithful if and only if the functor \eqref{eq-geom-to-cat-funct}:
$$
\Coh(\widetilde{\op{Jac}}{}_n^{\natural})^{\op{id}} \rightarrow \Coh(\widetilde{\op{Jac}}_n)^{\op{Ver}_*}
$$
is fully faithful. Let us again consider a more abstract setting. Suppose $S$ is a Noetherian base scheme and $\iota : Y\hookrightarrow X$ is a closed immersion of $S$-schemes. Let $v : X\rightarrow X$ be an endomorphism which induces the identity map on $Y$ (i.e.,~$v\iota = \iota$). In this context, it is possible to consider the functor:
\begin{equation}
\label{eq-coh-from-fixed-subscheme}
\Coh(Y)^{\op{id}} \rightarrow \Coh(X)^{v_*},
\end{equation}
sending $(\cal F, \alpha)$ to the the coherent sheaf $\iota_*\cal F$ together with the isomorphism:
\begin{equation}
\label{eq-endomorphism-pushforward}
\iota_*\cal F \xrightarrow{\iota_*\alpha} \iota_*\cal F \xrightarrow{\sim} v_*\iota_*\cal F.
\end{equation}
This construction includes \eqref{eq-geom-to-cat-funct} as a special case. Therefore, Theorem \ref{thm-nilpotence} is a consequence of the following general result.

\begin{prop}
\label{prop-nilpotence-abstract}
Suppose $\iota : Y\hookrightarrow X$ is a regular closed immersion of Noetherian schemes and $Y$ is affine. Suppose $v : X\rightarrow X$ is an endomorphism which induces the identity map on $Y$. Then the following are equivalent:
\begin{enumerate}[(a)]
	\item the functor \eqref{eq-coh-from-fixed-subscheme} is fully faithful;
	\item the endomorphism on $N_{Y/X}$ induced by $v$ is nilpotent.
\end{enumerate}
\end{prop}

\noindent
The rest of this section is devoted to the proof of Proposition \ref{prop-nilpotence-abstract}. It will follow from an analysis of the Koszul resolution of $\iota_*\cal O_Y$.

\subsection{Proof of Proposition \ref{prop-nilpotence-abstract}}

\subsubsection{}
Throughout this subsection, we put ourselves in the context of Proposition \ref{prop-nilpotence-abstract}. Our first goal is to reformulate the fully faithfulness in more explicit terms. For the purpose of the proof, we need to consider lax-fixed points (see \S\ref{sec-laxed-fixed-points}). Let $(\cal F, \alpha)$ be an object in $\Coh(Y)^{\op{id}\rightarrow\op{id}}$ and $(\cal G, \beta)\in\Coh(Y)^{\op{id}}$. Their Hom-space in $\Coh(Y)^{\op{id}\rightarrow\op{id}}$ is calculated as the fiber of the map:
$$
\xysmall{
	 \op{Hom}_{\Coh(Y)}(\cal F, \cal G) \ar[d]^{\op{id} - \beta^{-1}\cdot(-) \cdot\alpha} \\
	 \op{Hom}_{\Coh(Y)}(\cal F, \cal G)
}
$$
(see Lemma \ref{lem-fixed-point-hom}). On the other hand, the Hom-space of their images in $\Coh(X)^{\op{id}\rightarrow v_*}$ is calculated as the fiber of the map:
\begin{equation}
\label{eq-hom-after-pushforward}
\xysmall{
	\op{Hom}_{\Coh(X)}(\iota_*\cal F, \iota_*\cal G) \ar[d]^{\op{id} - (\iota_*\beta)^{-1}\cdot v_*(-)\cdot \iota_*\alpha} \\
	\op{Hom}_{\Coh(X)}(\iota_*\cal F, \iota_*\cal G)
}
\end{equation}
Here, the notation $\iota_*\alpha$ (resp.~$\iota_*\beta$) is slightly abused to denote the composition \eqref{eq-endomorphism-pushforward}.

\subsubsection{}
It is possible to rewrite \eqref{eq-hom-after-pushforward} as an endomorphism of $\op{Hom}_{\Coh(Y)}(\iota^*\iota_*\cal F, \cal G)$ using the adjunction between $\iota^*$ and $\iota_*$.\footnote{Notation: $\iota^*$ is the \emph{derived} inverse image functor.} More precisely, for any $\cal F \in \Coh(Y)$, we define a canonical endomorphism on $\iota^*\iota_*\cal F$:
\begin{equation}
\label{eq-endomorphism-on-push-pull}
v_{\cal F} : \iota^*\iota_*\cal F \rightarrow \iota^*\iota_*\cal F,
\end{equation}
given by composing the canonical identification $\iota^*\iota_*\cal F \xrightarrow{\sim} \iota^*v_*\iota_*\cal F$ with the map:
$$
\iota^*v_*\cal E \rightarrow \iota^*\cal E,
$$
defined naturally for any quasi-coherent sheaf $\cal E$ on $X$ by adjunction and the equality $v\iota = \iota$, and which is then applied to $\cal E = \iota_*\cal F$. By a straightforward (but tedious) calculation, the morphism \eqref{eq-hom-after-pushforward} is isomorphic to:
$$
\xysmall{
	\op{Hom}_{\Coh(Y)}(\iota^*\iota_*\cal F, \cal G) \ar[d]^{\op{id} - \beta^{-1}\cdot (-)\cdot v_{\cal F}\cdot \iota^*\iota_*\alpha} \\
	\op{Hom}_{\Coh(Y)}(\iota^*\iota_*\cal F, \cal G)
}
$$

\subsubsection{}
In conclusion, the full faithfulness of \eqref{eq-coh-from-fixed-subscheme} is equivalent to the Cartesian-ness of the following commutative square:
\begin{equation}
\label{eq-fully-faithful-cartesian-adjoint}
\xysmall{
	\op{Hom}_{\Coh(Y)}(\cal F, \cal G) \ar[d]^{\op{id} - \beta^{-1}\cdot(-) \cdot\alpha}\ar[r] & \op{Hom}_{\Coh(Y)}(\iota^*\iota_*\cal F, \cal G) \ar[d]^{\op{id} - \beta^{-1}\cdot (-)\cdot v_{\cal F}\cdot \iota^*\iota_*\alpha} \\
	 \op{Hom}_{\Coh(Y)}(\cal F, \cal G) \ar[r] & \op{Hom}_{\Coh(Y)}(\iota^*\iota_*\cal F, \cal G)
}
\end{equation}
when $(\cal F, \alpha)$, $(\cal G, \beta)$ both belong to $\Coh(Y)^{\op{id}}$. We will use this formulation of condition (a) in our proof of Proposition \ref{prop-nilpotence-abstract}. Here is a roadmap of the proof to follow.
\begin{enumerate}[(a)]
	\item We will first calculate $v_{\cal F}$ for $\cal F = \cal O_Y$.
	\item Next, we will prove ``(a) $\Rightarrow$ (b)" using the assumption on the Cartesian-ness of \eqref{eq-fully-faithful-cartesian-adjoint} only for $(\cal F, \alpha) = (\cal O_Y, \op{id})$; this part of the proof will be completed in \S\ref{sec-cartesian-square-implies-nilpotence}.
	\item Finally, we will prove ``(b) $\Rightarrow$ (a)" by first establishing it for $\cal F$ being finite locally free (essentially using the calculation of $v_{\cal F}$ for $\cal F = \cal O_Y$) and then reduces to this case by taking ``free resolutions" of an arbitrary $(\cal F, \alpha)$; the proof will be completed in \S\ref{sec-nilpotence-implies-cartesian-square}.
\end{enumerate}

\begin{lem}
\label{lem-wedge-power-to-cohomology}
There is a canonical isomorphism of $\cal O_Y$-modules for each $i\ge 0$:
\begin{equation}
\label{eq-wedge-power-to-cohomology}
\bigwedge^i N_{Y/X} \xrightarrow{\sim} \op H^{-i}(\iota^*\iota_*\cal O_Y),
\end{equation}
which renders the following diagram commutative
\begin{equation}
\label{eq-wedge-power-to-cohomology-endomorphism}
\xysmall{
	\bigwedge^i N_{Y/X} \ar[d]^{\bigwedge^iv} \ar[r]^-{\cong} & \op H^{-i}(\iota^*\iota_*\cal O_Y) \ar[d]^{\op H^{-i}v_{\cal O_Y}} \\
	\bigwedge^i N_{Y/X} \ar[r]^-{\cong} & \op H^{-i}(\iota^*\iota_*\cal O_Y)
}
\end{equation}
\end{lem}
\begin{proof}
Let $\cal I$ denote the ideal sheaf of $Y$. The exact sequence:
\begin{equation}
\label{eq-ideal-exact-sequence}
0 \rightarrow \cal I \rightarrow \cal O_X \rightarrow \iota_*\cal O_Y \rightarrow 0
\end{equation}
induces an isomorphism between $\op H^{-1}\iota^*\iota_*\cal O_Y$ and $\op H^0\iota^*\cal I$, which is by definition $N_{Y/X}$. In other words, we have found the isomorphism $N_{Y/X}\xrightarrow{\sim} \op H^{-1}(\iota^*\iota_*\cal O_Y)$. This map renders \eqref{eq-wedge-power-to-cohomology-endomorphism} commutative for $i=1$. The maps \eqref{eq-wedge-power-to-cohomology} for $i\neq 1$ are constructed using the graded commutative algebra structure on $\op H^{\bullet}(\iota^*\iota_*\cal O_Y)$. The fact that they define isomorphisms making \eqref{eq-wedge-power-to-cohomology-endomorphism} commutative follows from the Koszul resolution of $\iota_*\cal O_Y$ as an $\cal O_X$-algebra.
\end{proof}

\subsubsection{}
We will now turn to a key step in the proof of Proposition \ref{prop-nilpotence-abstract}. Namely, we will show that the fully faithfulness statement applied to $(\cal F, \alpha) = (\cal O_Y, \op{id})$ implies the nilpotence of the $v$-action on $N_{Y/X}$. We begin with a commutative algebra fact.

\begin{lem}
\label{lem-commutative-algebra-nilpotence}
Suppose $R$ is a Noetherian ring and $\cal M$ is a finite locally free $R$-module equipped with an endomorphism $v$. If for all finite $R$-module $\cal N$ together with an automorphism $\gamma$, the endomorphism:
$$
\op{id} - \gamma \otimes v : \cal N \otimes \cal M \rightarrow \cal N\otimes\cal M
$$
is invertible, then $v$ is nilpotent.
\end{lem}
\begin{proof}
Without loss of generality, we may assume that $\Spec(R)$ is connected and $\cal M$ has constant rank $r$. We first show that $v$ is \emph{not} invertible. Suppose the contrary. Let $\cal N := \cal M^{\vee}$ together with the automorphism $\gamma := (v^{-1})^{\vee}$. We have a commutative diagram:
$$
\xysmall{
	\cal M^{\vee} \otimes \cal M \ar[r]\ar[d]^{\op{id} - (v^{-1})^{\vee}\otimes v} & \op{End}(\cal M) \ar[d]^{\op{id} - v(-)v^{-1}} \\
	\cal M^{\vee} \otimes \cal M \ar[r] & \op{End}(\cal M)
}
$$
where the right vertical arrow annihilates $\op{id}_{\cal M}$. Contradiction. We will now prove that $v$ is nilpotent for increasingly general rings $R$.

First, we suppose that $R$ is a field. In this case, $\cal M$ is a finite-dimensional $R$-vector space. Hence $v$ is not injective. Let $\cal M_1\neq 0$ be its kernel and $\cal Q_1 := \cal M/\cal M_1$ be the quotient. The endomorphism $v$ induces an endomorphism $v_1$ on $\cal Q_1$. For any $(\cal N, \gamma)$ as in the statement of the Lemma, we have a commutative diagram with exact rows:
	$$
	\xysmall{
		\cal N \otimes \cal M_1 \ar[r]\ar[d]^{\op{id}} & \cal N\otimes \cal M \ar[r]\ar[d]^{\op{id} - \gamma\otimes v} & \cal N\otimes\cal Q_1 \ar[r]\ar[d]^{\op{id} - \gamma \otimes v_1} & 0 \\
		\cal N \otimes \cal M_1 \ar[r] & \cal N\otimes \cal M \ar[r] & \cal N \otimes\cal Q_1 \ar[r] & 0
	}
	$$
	Since the middle arrow is invertible, so is the right arrow. Hence $(\cal Q_1, v_1)$ again satisfies the hypothesis of the Lemma. This shows that we may inductively build a filtration:
	$$
	0 \subset \cal M_1 \subset \cal M_2 \subset \cdots \subset \cal M,
	$$
	where $v(\cal M_i) \subset \cal M_{i-1}$. Since $\cal M$ has finite dimension, this filtration terminates at finite step and we conclude that $v$ is nilpotent.
	
Now, we suppose that $R$ is reduced and Noetherian. For any $\fr p \in \Spec(R)$, the base change of $(\cal M, v)$ to $\kappa(\fr p)$ satisfies the hypothesis of the Lemma (with $R$ replaced by $\kappa(\fr p)$). Hence the previous step shows that $v^r$ reduces to zero at each fiber $R\rightarrow \kappa(\fr p)$ (where $r$ is the rank of $\cal M$). Let $\cal M'$ denote the cokernel of the map
	$$
	\cal M \xrightarrow{v^r} \cal M \rightarrow \cal M' \rightarrow 0.
	$$
	Then $\cal M'$ has constant fibers of dimension $r$. Since $R$ is reduced, it follows that $\cal M'$ is locally free of rank $r$. The surjection $\cal M \rightarrow \cal M'$ must therefore be an isomorphism. Hence $v^r = 0$.
	
Finally, we suppose that $R$ is any Noetherian ring. The previous step shows that the endomorphism $v^r : \cal M \rightarrow \cal M$ factors through $\fr n \cal M$, where $\fr n$ is the nilradical of $R$. Since $R$ is Noetherian, $\fr n$ is nilpotent. Hence we must have $v^{r'}(\cal M) = 0$ for some integer $r' \ge r$.
\end{proof}

\subsubsection{}
\label{sec-cartesian-square-implies-nilpotence}
The implication ``(a) $\Rightarrow$ (b)" in Proposition \ref{prop-nilpotence-abstract} is a consequence of the following.

\begin{lem}
\label{lem-cartesian-square-implies-nilpotence}
If for every object $(\cal G, \beta) \in \Coh(Y)^{\op{id}}$, the commutative square \eqref{eq-fully-faithful-cartesian-adjoint} for $(\cal F, \alpha) = (\cal O_Y, \op{id})$ is Cartesian, then the $v$-action on $N_{Y/X}$ is nilpotent.
\end{lem}
\begin{proof}
Since $\cal O_Y$ is identified with $\op H^0\iota^*\iota_*\cal O_Y$ via the co-unit map, the assumption implies (is in fact equivalent to) that the endomorphism $\op{id} - \beta^{-1}\cdot(-)\cdot\tau^{\le -1}v_{\cal O_Y}$ acting on the space $\op{Hom}_{\Coh(Y)}(\tau^{\le -1}\iota^*\iota_*\cal O_Y, \cal G)$ is an equivalence. Taking $\cal G$ to be any coherent sheaf $\cal N$ concentrated in cohomological degree $-1$, this assertion implies that:
$$
\op{id} - \beta^{-1}\cdot(-)\cdot v : \op{Hom}_{\Coh(Y)^{\heartsuit}}(N_{Y/X}, \cal N) \rightarrow \op{Hom}_{\Coh(Y)^{\heartsuit}}(N_{Y/X}, \cal N)
$$
is bijective; here we have used Lemma \ref{lem-wedge-power-to-cohomology} to identify $\op H^{-1}\iota^*\iota_*\cal O_Y$ with $N_{Y/X}$ in a way compatible with the $v$-action. Since:
$$
\op{Hom}_{\Coh(Y)^{\heartsuit}}(N_{Y/X}, \cal N) \xrightarrow{\sim} \cal N\otimes N_{Y/X}^{\vee},
$$
we may apply Lemma \ref{lem-commutative-algebra-nilpotence} to $(\cal M, v) = (N^{\vee}_{Y/X}, v^{\vee})$ to conclude.
\end{proof}

\subsubsection{}
\label{sec-nilpotence-implies-cartesian-square}
We now turn to the ``(b) $\Rightarrow$ (a)" direction in Proposition \ref{prop-nilpotence-abstract}. Namely, we will assume that the endomorphism on $N_{Y/X}$ induced from $v$ is nilpotent. We will prove a stronger statement: the commutative square \eqref{eq-fully-faithful-cartesian-adjoint} is Cartesian for all $(\cal F, \alpha)\in\Coh(Y)^{\op{id}\rightarrow\op{id}}$ and $(\cal G, \beta)\in\Coh(Y)^{\op{id}}$.

\begin{lem}
\label{lem-nilpotence-implies-fully-faithfulness-free}
Under the assumptions of \S\ref{sec-nilpotence-implies-cartesian-square}, the commutative square \eqref{eq-fully-faithful-cartesian-adjoint} is Cartesian if $\cal F$ is a finite locally free $\cal O_Y$-module.
\end{lem}
\begin{proof}
Since the problem is local on $Y$, we may assume that $\cal F$ is a free $\cal O_Y$-module of rank $k$. Because $\cal O_Y$ is identified with $\op H^0\iota^*\iota_*\cal O_Y$ via the co-unit map, it suffices to show that the endomorphism:
\begin{equation}
\label{eq-endomorphism-truncated}
\xysmall{
	\op{Hom}_{\Coh(Y)}(\tau^{\le -1}\iota^*\iota_*\cal F, \cal G)\ar[d]^{\op{id} - \beta^{-1}\cdot(-)\cdot \tau^{\le -1}v_{\cal F}\cdot\tau^{\le -1}\iota^*\iota_*\alpha} \\
	\op{Hom}_{\Coh(Y)}(\tau^{\le -1}\iota^*\iota_*\cal F, \cal G)
}
\end{equation}
is an isomorphism. Since $\iota^*\iota_*\cal F$ is cohomological bounded, via cohomological truncations, it suffices to prove that for each $i\ge 1$, the following map:
\begin{equation}
\label{eq-endomorphism-truncated-single-degree}
\xysmall{
	\op{Hom}_{\Coh(Y)}(\op H^{-i}(\iota^*\iota_*\cal F), \cal G)\ar[d]^{\op{id} - \beta^{-1}\cdot(-)\cdot \op H^{-i}v_{\cal F}\cdot \op H^{-i}\iota^*\iota_*\alpha} \\
	\op{Hom}_{\Coh(Y)}(\op H^{-i}(\iota^*\iota_*\cal F), \cal G)
}
\end{equation}
is invertible. Since $v_{\cal F}$ is functorially assigned to $\cal F$, the endomorphisms $v_{\cal F}$ and $\iota^*\iota_*\alpha$ commute. Thus, so do their truncated versions. It follows that the $n$-fold composition of $\beta^{-1}\cdot(-)\cdot \op H^{-i}v_{\cal F}\cdot \op H^{-i}\iota^*\iota_*\alpha$ is given by:
$$
\beta^{-n}\cdot (-)\cdot (\op H^{-i}v_{\cal F})^n \cdot (\op H^{-i}\iota^*\iota_*\alpha)^n.
$$
Since $\op H^{-i}v_{\cal F}$ is identified with $\op H^{-i} v_{\cal O_Y}^{\oplus k}$ acting on $\op H^{-i}\iota^*\iota_*\cal O_Y^{\oplus k}$, the hypothesis on nilpotence and Lemma \eqref{lem-wedge-power-to-cohomology} shows that $(\op H^{-i}v_{\cal F})^n$ (for $i\ge 1$) vanishes for $n$ sufficiently large. This implies that \eqref{eq-endomorphism-truncated-single-degree} is invertible.
\end{proof}

\subsubsection{}
Finally, we will reduce the general statement to the case where $\cal F$ is a finite locally free $\cal O_Y$-module. For this reduction, it is convenient to introduce the cocomplete $\infty$-category:
$$
\QCoh(Y)^{\op{id}\rightarrow\op{id}}_{\op{loc.fin}}\subset \QCoh(Y)^{\op{id}\rightarrow\op{id}},
$$
consisting of objects $(\cal M, \beta)$ such that for each $i$, the action of $\op H^i(\beta)$ on $\op H^i(\cal M)$ is \emph{locally finite}, i.e., $\op H^i(\cal M)$ is a union of finite submodules stable under $\op H^i(\beta)$. It is clear that the full subcategory $\Coh(Y)^{\op{id}\rightarrow\op{id}} \subset \QCoh(Y)^{\op{id}\rightarrow\op{id}}$ is contained in the ``locally finite" subcategory.

\begin{lem}
The stable $\infty$-category $\QCoh(Y)_{\op{loc.fin}}^{\op{id}\rightarrow\op{id}}$ is generated under colimits by objects of the form $(\cal F, \alpha)$ where $\cal F$ is a finite locally free $\cal O_Y$-module up to cohomological shift.
\end{lem}
\begin{proof}
It suffices to show that every nonzero object $(\cal M, \beta)$ of $\QCoh(Y)^{\op{id} \rightarrow \op{id}}_{\op{loc.fin}}$ receives a nonzero map from some $(\cal F, \alpha)$ as in the statement of the Lemma. Suppose $\op H^i(\cal M) \neq 0$ for some $i$. We will construct such a pair $(\cal F, \alpha)$ equipped with a map to $(\tau^{\le i}\cal M, \tau^{\le i}\beta)$ such that the induced map on $\op H^i$ is nonzero. The $\infty$-category $\QCoh(Y)^{<\infty}$ is the differential graded nerve of bounded above complexes of projective modules (c.f.~\cite[1.3.2.7]{Lu17}). Hence we may assume that $\tau^{\le i}\cal M$ is such a complex and $\beta$ is represented by a chain map. Hence $(\tau^{\le i}\cal M, \tau^{\le i}\beta)$ is of the form:
$$
\xysmall{
	\cdots \ar[r] & M^{i-2} \ar[r]^{\delta^{i-2}} \ar[d]^{\beta^{i-2}} & \cal M^{i-1} \ar[r]^{\delta^{i-1}} \ar[d]^{\beta^{i-1}} & \op{Ker}(\delta^i) \ar[r]\ar[d]^{\beta^i} & 0 \ar[r]\ar[d] & \cdots \\
	\cdots \ar[r] & M^{i-2} \ar[r]^{\delta^{i-2}} & \cal M^{i-1} \ar[r]^{\delta^{i-1}} & \op{Ker}(\delta^i) \ar[r] & 0 \ar[r] & \cdots
}
$$
where each $\cal M^j$ is projective. Since the action of $\op H^i(\beta)$ on $\op H^i(\cal M)$ is locally finite, we can find a finite submodule $\cal N \subset \op H^i(\cal M)$ stable under $\op H^i(\beta)$. Then there exists a surjection from some finite locally free $\cal O_Y$-module $f : \cal F \twoheadrightarrow \cal N$, and an endomorphism $\alpha$ of $\cal F$ making the following diagram commute:
$$
\xysmall{
	\cal F \ar[d]^{\alpha}\ar@{->>}[r]^-f & \cal N \ar[d]^{\op H^i(\beta)}  \\
	\cal F \ar@{->>}[r]^-f & \cal N 
}
$$
Choose any lift of $f$ to a map $\tilde f : \cal F \rightarrow \op{Ker}(\delta^i)$. Then $\tilde f$ defines a chain map $\cal F[-i] \rightarrow \tau^{\le i}\cal M$. Furthermore, the two maps $\beta^i\tilde f$ and $\tilde f\alpha$ from $\cal F$ to $\op{Ker}(\delta^i)$ differ by a morphism $\cal F \rightarrow \op{Im}(\delta^{i-1})$, which can be lifted to a map $h : \cal F \rightarrow \cal M^{i-1}$. The map $h$ witnesses the commutativity of the following diagram in $\QCoh(Y)^{<\infty}$:
$$
\xysmall{
	\cal F[-i] \ar[d]^{{\alpha[-i]}}\ar[r] & \tau^{\le i}\cal M \ar[d]^{\tau^{\le i}\beta} \\
	\cal F[-i] \ar[r] & \tau^{\le i}\cal M
}
$$
Hence we find a morphism from $(\cal F, \alpha)[-i]$ to $(\tau^{\le i}\cal M, \tau^{\le i}\beta)$ in $\QCoh(Y)^{\op{id}\rightarrow\op{id}}_{\op{loc.fin}}$ which induces a nonzero map on $\op H^i$.
\end{proof}

\subsubsection{}
To prove ``(b) $\Rightarrow$ (a)" in Proposition \ref{prop-nilpotence-abstract}, one observes that the commutative square associated to $(\cal F, \alpha)\in\QCoh(Y)^{\op{id}\rightarrow\op{id}}_{\op{loc.fin}}$ and $(\cal G, \beta)\in\Coh(Y)^{\op{id}}$:
$$
\xysmall{
	\op{Hom}_{\QCoh(Y)}(\cal F, \cal G) \ar[d]^{\op{id} - \beta^{-1}\cdot(-) \cdot\alpha}\ar[r] & \op{Hom}_{\QCoh(Y)}(\iota^*\iota_*\cal F, \cal G) \ar[d]^{\op{id} - \beta^{-1}\cdot (-)\cdot v_{\cal F}\cdot \iota^*\iota_*\alpha} \\
	 \op{Hom}_{\QCoh(Y)}(\cal F, \cal G) \ar[r] & \op{Hom}_{\QCoh(Y)}(\iota^*\iota_*\cal F, \cal G)
}
$$
defined analogously to \eqref{eq-fully-faithful-cartesian-adjoint}, takes colimits in the object $(\cal F, \alpha) \in \QCoh(Y)^{\op{id} \rightarrow \op{id}}_{\op{loc.fin}}$ to limits of commutative squares. Hence the problem reduces to the case where $\cal F$ is a finite locally free $\cal O_Y$-module up to cohomological shift, where it follows from Lemma \ref{lem-nilpotence-implies-fully-faithfulness-free}.

\bigskip

\end{document}